\RequirePackage{fix-cm}
\documentclass[smallextended]{svjour3}       % onecolumn (second format)
\smartqed  % flush right qed marks, e.g. at end of proof
\usepackage{graphicx}
 \usepackage{mathptmx}      % use Times fonts if available on your TeX system
\usepackage{latexsym}
\usepackage{array}
\usepackage{epsf}
\usepackage{amsmath,bm}
\usepackage{amsfonts}
\usepackage[square, comma, sort&compress, numbers]{natbib}
\usepackage{fancyhdr}
\usepackage{pxfonts}
\usepackage{caption}
\usepackage{epstopdf}
\usepackage{algorithm}
\usepackage{empheq}
\usepackage{stmaryrd}
\usepackage{color}

\newcommand{\bc}{\begin{center}}
\newcommand{\ec}{\end{center}}
\newcommand{\be}{\begin{eqnarray}}
\newcommand{\ee}{\end{eqnarray}}

\newcommand{\ben}{\begin{eqnarray*}}
\newcommand{\een}{\end{eqnarray*}}

\newcommand{\Om}{{\rm\Omega}}

\newcommand{\Rmnum}[1]{\expandafter\@slowromancap\romannumeral #1@}

\newcommand{\cE}{{\mathcal{E}_h}}
\newcommand{\cEi}{{\mathcal{E}_h^I}}

\newcommand{\cED}{{\mathcal{E}_h^D}}
\newcommand{\cEN}{{\mathcal{E}_h^N}}

\newcommand{\cO}{\mathcal{O}}

\newcommand{\cTh}{{\mathcal{T}_h}}
\newcommand{\PcT}{{\partial\mathcal{T}_h}}
\newcommand{\R}{\mathbb{R}}
\def\S{\mathcal{S}}

\newcommand{\dsig}{\boldsymbol{\sigma}_h}
\newcommand{\dtau}{\boldsymbol{\tau}_h}
\newcommand{\hsig}{\check{\boldsymbol{\sigma}}_h}
\newcommand{\htau}{\check{\boldsymbol{\tau}}_h}
\newcommand{\drh}{\boldsymbol{r}_h}

\newcommand{\ttsig}{\tilde{\boldsymbol{\sigma}}_h}
\newcommand{\tttau}{\tilde{\boldsymbol{\tau}}_h}

\newcommand{\basig}{\bar{\boldsymbol{\sigma}}_h}
\newcommand{\batau}{\bar{\boldsymbol{\tau}}_h}

\newcommand{\tsig}{\hat{\boldsymbol{\sigma}}_h}
\newcommand{\tu}{\hat{u}_h}
\newcommand{\ttau}{\hat{\boldsymbol{\tau}}_h}
\newcommand{\tv}{\hat{v}_h}
\newcommand{\csig}{\boldsymbol{\sigma}}
\newcommand{\ctau}{\boldsymbol{\tau}}

\newcommand{\bn}{n}

\newcommand{\hu}{\check u_h}
\newcommand{\du}{u_h}
\newcommand{\hv}{\check v_h}
\newcommand{\dv}{v_h}
\newcommand{\divh}{{\rm div} }
\begin{document}

\title{New Discontinuous Galerkin Algorithms and Analysis for Linear Elasticity
with Symmetric Stress Tensor
\thanks{
The work of Hong, Ma and Xu was partially supported by 
Center for Computational Mathematics and Applications, The Pennsylvania State University. The work of Hu was supported by NSFC projects 11625101.}
}


\author{Qingguo Hong       \and
Jun Hu  \and
Limin Ma    \and
Jinchao Xu %etc.
}


\institute{Qingguo Hong \at
Department of Mathematics, Pennsylvania State University, University Park, PA, 16802, USA. \\
\email{huq11@psu.edu}           %  \\
           \and
           Jun Hu \at
           School of Mathematical Science, Peking University, Beijing 100871, P. R. China. \\
           \email{hujun@math.pku.edu.cn} 
           \and
           Limin Ma \at
           Department of Mathematics, Pennsylvania State University, University Park, PA,
16802, USA. \\
\email{lum777@psu.edu}
\and
Jinchao Xu \at
Department of Mathematics, Pennsylvania State University, University Park, PA,
16802, USA.\\
\email{xu@math.psu.edu}
}

\date{Received: date / Accepted: date}

\maketitle

\begin{abstract}
This paper presents a new and unified approach to the derivation and analysis of many existing, as well as new discontinuous
Galerkin methods  for linear elasticity problems.  The analysis is based on a
unified discrete formulation for the linear elasticity
problem consisting of four discretization variables: strong symmetric
stress tensor $\dsig$ and displacement $\du$ inside each element, and the
modifications of these two variables $\hsig$ and $\hu$ on elementary
boundaries of elements.  
Motivated by many relevant methods in the literature,
this formulation can be used to derive most existing discontinuous, nonconforming and
conforming Galerkin methods for linear elasticity problems and
especially to develop a number of new discontinuous Galerkin methods.
Many special cases of this four-field  formulation are proved to be
hybridizable and can be reduced to some known
hybridizable discontinuous Galerkin, weak Galerkin and local
discontinuous Galerkin methods by eliminating one or two of the four
fields.  As certain stabilization parameter tends to zero, this
four-field formulation is proved to converge to some conforming and
nonconforming mixed methods for linear elasticity problems.
Two families of inf-sup conditions, one known as $H^1$-based and
the other known as $H({\rm div})$-based, are proved to be uniformly
valid with respect to different choices of discrete spaces and
parameters. These inf-sup conditions guarantee the well-posedness of
the new proposed methods and also offer a new and unified analysis for
many existing methods in the literature as a by-product. Some numerical examples are provided to verify the theoretical analysis including the optimal convergence of the new proposed methods.
\end{abstract}

\section{Introduction}
\label{intro} 

In this paper, we introduce a unified formulation and analysis for linear elasticity problems 
\begin{equation}\label{model}
\left\{
\begin{aligned}[rll]
A\csig-\epsilon (u)&=0\quad &\text{ in }\Om ,\\
\divh \csig&=f\quad & \text{ in }\Om ,\\
u&=0\quad &\text{ on }\Gamma_D ,\\
\csig\bn&=0\quad &\text{ on }\Gamma_N,
\end{aligned}
\right.
\end{equation}
with $\Om\subset \R^n~(n=2,3)$ and $\partial \Om=\Gamma_D\cup \Gamma_N$, $\Gamma_D\cap \Gamma_N=\varnothing$. %In this paper, we introduce a unified formulation and analysis for this model problem with $\partial \Omega=\Gamma_D$ and $g_D=0$.
 Here the displacement is denoted by $u: \Om\rightarrow \R^n$ and the stress tensor is denoted by $\csig: \Om\rightarrow \S$, where $\S$ is the set of symmetric $n\times n$ tensors. The linearized strain tensor $\epsilon (u)=\frac{1}{2}(\nabla u+\nabla u^T)$. The compliance tensor $A: \S\rightarrow \S$
\iffalse
\begin{equation}\label{lame}
A\csig = \frac{1}{2\mu}(\csig -\frac{\lambda}{2\mu+n\lambda }tr(\csig)I),\ \lambda>0,\ \mu>0
\end{equation}
is assumed to be bounded and symmetric positive definite, where $\lambda$ and $\mu$ are the Lam$\rm \acute{e}$ coefficients of the elastic material under consideration.  
\fi 
\begin{equation}\label{lame}
A\csig ={1+\nu\over E}\csig -{(1+\nu)\nu\over (1+(n-2)\nu) E}tr(\csig)I
\end{equation}
is assumed to be bounded and symmetric positive definite, where $E$ and $\nu\in (0, \frac12)$ are the Young's modulus and the Poisson's ratio of the elastic material under consideration, respectively.

Finite element method (FEM) and its variants have been widely used for numerical solutions of partial differential equations. Conforming and nonconforming FEMs in primal form are two classic Galerkin methods for elasticity and structural problems \cite{hrennikoff1941solution,courant1994variational,feng1965finite}. Mixed FEMs for the elasticity problem, derived from the Hellinger-Reissner variational principle, are also popular methods since they %avoid locking effects and 
approximate not only the displacement  but also the stress tensor. Unlike the mixed FEMs for scalar second-order elliptic problems, the strong symmetry is required for the stress tensor in the elasticity problem. This strong symmetry causes a substantial additional difficulty for developing stable mixed FEMs for the elasticity problem.
To overcome such a difficulty, it was proposed in \cite{Fraejis1975} to relax or abandon the symmetric constraint on the stress tensor by employing Lagrangian functionals. This idea was developed
 in late nineteens \cite{amara1979equilibrium, arnold1984peers, arnold1988new, stein1990mechanical,
 stenberg1986construction, stenberg1988family, stenberg1988two}, and further systematically explored in a recent work \cite{arnold2007mixed} by utilizing a constructive derivation of the elasticity complex starting from the de Rham complex \cite{eastwood2000complex} and mimicking the construction in the discrete case.
Another framework to construct stable weakly symmetric mixed finite elements was presented in \cite{boffi2009reduced}, where two approaches were particularly proposed with the first one based on {the Stokes problem} and the second one based on interpolation operators. To keep the symmetry of discrete stress, a second way is to relax the continuity of the normal components of discrete stress across the internal edges or faces of grids. This approach leads to nonconforming mixed FEMs with strong symmetric stress tensor \cite{yi2005nonconforming,yi2006new,man2009lower,hu2007lower,awanou2009rotated,arnold2003nonconforming,
arnold2014nonconforming,gopalakrishnan2011symmetric,wu2017interior,hu2019nonconforming}.
In 2002, based on the elasticity {complex}, the first family of symmetric conforming mixed elements with polynomial shape functions was proposed for {the two-dimensional case} in \cite{arnold2002mixed}, which was extended to { the } three-dimensional { case} in \cite{arnold2008finite}.
Recently, a family of conforming mixed elements with fewer degrees of freedom was proposed for any dimension by discovering a
 crucial structure of discrete stress spaces of symmetric matrix-valued polynomials on any dimensional simplicial grids and proving two basic algebraic results in \cite{hu2014family,hu2015family,hu2016finite,hu2014finite}. Those new elements can be regarded as an improvement and a unified extension to any dimension of those from \cite{arnold2002mixed} and \cite{arnold2008finite}, without an explicit use of the elasticity complex. Besides the optimal convergence property with respect to the degrees of polynomials of discrete stresses, an advantage of those elements is that it is easy to construct their basis functions, therefore implement the elements. See stabilized mixed finite elements on  simplicial grids for any dimension in \cite{chen2017stabilized}.
 
Discontinuous Galerkin (DG) methods were also widely used in numerical solutions for the elasticity problem, see \cite{chen2010local,hong2016robust, hong2019conservative, wang2020mixed}. DG methods offer the convenience to discretize problems in an element-by-element fashion and use numerical traces to glue each element together \cite{arnold2002unified, hong2012discontinuous, hong2016uniformly, hong2018parameter}. This advantage makes DG methods an ideal option for linear elasticity problems to preserve the strong symmetry of the stress tensor.
Various hybridizable discontinuous Galerkin (HDG) formulations with strong symmetric stress tensor were proposed and analyzed for linear elasticity problems, such as \cite{soon2008hybridizable,soon2009hybridizable,fu2015analysis,qiu2018hdg,chen2016robust}. The HDG methods for linear elasticity problems contain three variables -- stress $\dsig$, displacement $\du$ and numerical trace of displacement $\hat u_h$. In the HDG methods, the variable $\hat u_h$ is defined on element borders and can be viewed as the Lagrange multiplier for the continuity of the normal component of stress. Weak Galerkin (WG) methods were proposed and analyzed in \cite{wang2016locking,wang2018weak,wang2018hybridized,chen2016robust,yi2019lowest} for linear elasticity problems. The main feature of the WG methods is the weakly defined differential operators over weak functions.
A three-field decomposition method was discussed for linear elasticity problems in \cite{brezzi2005three}. A new hybridized mixed method for linear elasticity problems was proposed in~\cite{gong2019new}.

{Virtual element method is a new Galerkin scheme for the approximation of partial differential equation problems, and   admits the flexibility to deal with general polygonal and polyhedral meshes. 
Virtual element method is experiencing a growing interest towards structural mechanics problems, and has contributed a lot  to linear elasticity problems, see  \cite{da2013virtual,artioli2017stress,artioli2018family,dassi2020three} and the reference therein. Recently,  investigation of the possible interest in using virtual element method for traditional decompositions is presented in \cite{brezzi2021finite}. As shown in  \cite{brezzi2021finite}, virtual element method looks promising for high-order partial differential equations as well as Stokes and linear elasticity problems. Some other interesting methods, say the tangential-displacement normal-normal-stress method which is robust with respect both shear and volume locking, were considered in \cite{pechstein2011tangential,pechstein2018analysis}.
}

In this paper, a unified formulation is built up for linear elasticity problems following and modifying the ones in \cite{hong2020extended,hong2021extended} for scalar second-order elliptic problems. The formulation is given in terms of four discretization variables --- $\dsig$, $\hsig$, $\du$, $\hu$. The variables $\dsig$ and $\du$ approximate the stress tensor $\csig$ and displacement $u$ in each element, respectively. Strong symmetry of the stress tensor is guaranteed by the symmetric shape function space of the variable $\dsig$. The variables $\hsig$ and $\hu$ are the residual corrections to the average of $\dsig$ and $\du$ along interfaces of elements, respectively. They can also be viewed as multipliers to impose the inter-element continuity property of $\du$ and the normal component of $\dsig$, respectively.
The four variables in the formulation provide feasible choices of numerical traces, and therefore, the flexibility of recovering most existing FEMs for linear elasticity problems.
There exist two different three-field formulations by eliminating the variable $\hsig$ and $\hu$, respectively, and a two-field formulation by eliminating both. 
With the same choice of discrete spaces and parameters, these four-field, three-field, and two-field formulations are equivalent.
Moreover, some particular discretizations induced from the unified formulation are hybridizable and lead to the corresponding one-field formulation.

As shown in \cite{ hong2019unified, hong2018uniform, hong2020extended}, the analysis of the formulation is classified into two classes:  
$H^1$-based class and $H({\rm div})$-based class. Polynomials of a higher degree for the displacement than those for the stress tensor 
are employed 
for the $H^1$-based formulation  and the other way around for the $H({\rm div})$-based formulation. 
Both classes are proved to be well-posed under natural assumptions.
Unlike scalar second order elliptic problems, there is no stable symmetric $H({\rm div})$-conforming mixed finite elements in the literature that approximates the stress tensor by polynomials with degree not larger than $k$ and $k\leq n$.
This causes the difficulty to prove the inf-sup condition for the $H({\rm div})$-based formulation with $k\leq n$. The nonconforming element in \cite{wu2017interior} is employed here to circumvent this difficulty with the jump of the normal component of $\dsig$ embedded in the norm of the stress tensor $\dsig$.

The unified formulation is closely related to some mixed element methods. As some parameters approach  zero, some mixed element method{s} and primal method{s} can be proven to be the limiting cases of the unified formulation. In particular, both the nonconforming mixed element method in \cite{gopalakrishnan2011symmetric}  and the conforming mixed element  methods in \cite{hu2014finite,hu2014family,hu2015family} are some limiting cases of the  formulation.
The proposed four-field formulation is also closely related to most existing methods \cite{qiu2018hdg,chen2016robust,soon2009hybridizable,fu2015analysis,chen2010local,wang2020mixed}
for linear elasticity as  listed in the first three rows in Table \ref{tab:HDGexist},  and the first row  in Table \ref{tab:LDG1} and Table \ref{tab:LDG2}.
%as listed in Table \ref{tab:introexist}. 
More importantly, some new discretizations are derived from this formulation as listed in %the last row in Table \ref{tab:HDGexist},  and the last two rows  in Table \ref{tab:LDG1} and Table \ref{tab:LDG2}.
Table \ref{tab:intronew}. 
Under the unified analysis of the four-field formulation, all these new methods are well-posed and admit optimal error estimates.
In Table \ref{tab:intronew}, the first scheme is an $H^1$-based method and the following two schemes are $H({\rm div})$-based methods. The
 last scheme is a special case of the second one with $\gamma=0$ and $\eta=\tau^{-1}$. The last scheme is hybridizable and can be 
 written as a one-field formulation with only one globally-coupled variable. % $\hsig$ or $\hu$.
In fact, after the elimination of variable $\check \sigma_h$ and a transformation from variable $\check u_h$ to variable $\hat u_h$  in the last method of Table \ref{tab:intronew}, we obtain an optimal $H({\rm div})$-based HDG method.

The notation $\tau=\Omega(h_e^{-1})$ and $\tau=\Omega(h_e)$ in %Table \ref{tab:introexist} and 
Table \ref{tab:intronew}
 means there exist constants $c_0>0, C_0>0$ such that $c_0 h_e^{-1}\le \tau\le C_0 h_e^{-1}$ and $c_0 h_e\le \tau\le C_0 h_e$, respectively.  For $k\ge 0$,
\begin{equation}\label{Spaces}
\begin{aligned}
V^{k}_h&=\{v_h\in L^2(\Omega, \R^n): v_h|_{K}\in \mathcal{P}_k(K, \R^n), \forall
K\in \mathcal T_h \},\\
Q^k_h& = \{\dtau \in L^2(\Omega, \S):
\dtau|_K\in \mathcal{P}_k(K, \S), \forall K\in
\mathcal T_h \},\\
\check V^{k}_{h}&=\{\hv\in L^2(\cE, \R^n): v_h|_e\in \mathcal{P}_k(e, \R^n), \forall e\in \cE, \ \hv|_{\Gamma_D}=0 \},\\
 \check Q^k_{h}& {= \{\htau \in L^2(\cE, \S): \htau|_e\in \boldsymbol {\mathcal{P}}_k(e, \S),  \forall e\in \cE, \ \hsig\bn|_{\Gamma_N}=0 \},}
\end{aligned}
\end{equation}
where $\boldsymbol {\mathcal{P}}_k(K, \R^{n})$ and $\boldsymbol {\mathcal{P}}_k(e, \R^{n})$ are vector-valued  in $\R^n$ and each component is in the space of polynomials of degree at most $k$ on $K$ and $e$, respectively, and $\boldsymbol {\mathcal{P}}_k(K, \S)$ are symmetric tensor-valued functions in $\S$ and each component is in the space of polynomials of degree at most $k$ on $K$.

\iffalse
\begin{table}[!htbp] 
\scriptsize
\centering
\begin{tabular}{c|cccccccccccc}
\hline
&$\eta$& $\tau$& $Q_h$ & $V_h$ & $\check Q_h$& $\check{V}_h$ &$\|\csig-\dsig\|_0$  & $\|u-\du\|_0$& $\|\epsilon (u)-\epsilon_h (\du)\|_0$&$\|\divh_h (\csig-\dsig)\|_0$  &
\\
\hline
1&$\tau^{-1}$& $\Omega(h_e)$ &$ Q_h^k$ & $V_h^{k}$ & $\check Q_h^{k}$ &
$\check{V}_h^{k}$ &$k+\frac{1}{2}$ &$k+1$ & $k+\frac{1}{2}$&- & HDG in \cite{soon2009hybridizable,fu2015analysis}
\\
2&$\tau^{-1}$& $\Omega(h_e^{-1})$&$ Q_h^{k}$ & $V_h^{k+1}$ & $\check Q_h^{k}$ &$\check{V}_h^{k}$& $k+1$ & $k+2$&$k+1$  &- &HDG in \cite{qiu2018hdg,chen2016robust}
\\
3&$\tau^{-1}$& $\Omega(h_e^{-1})$&$ Q_h^k$ & $V_h^{k}$ &$\check Q_h^{k}$ &
$\check{V}_h^k$& $k$& $k+1$ &$k$&- & HDG in \cite{soon2009hybridizable}
\\
4&0& $\Omega(h_e^{-1})$&$ Q_h^k$ & $V_h^{k+1}$& $\check Q_h^{k}$ &
$\check{V}_h^k$ & $k+1$ & $k+2$  &$k+1$  & - &LDG in \cite{chen2010local}
\\\hline
5&$\Omega(h_e^{-1})$& 0& $ Q_h^{k+1}$  &$V_h^{k}$&
$\check Q_h^{k}$ &$\check{V}_h^{k+1}$ & $k+1$ &$k+1$ &- &$k+1$ & LDG in \cite{wang2020mixed}
\\\hline
\end{tabular}
\caption{\footnotesize{Some existing methods in the literature with $\gamma=0$. Here $\eta$ are $\tau$ are to enhance the continuity of the stress tensor $\dsig$ and displacement $\du$, respectively.}}
\label{tab:introexist}
\end{table}
\fi

\begin{table}[!htbp] 
\centering
\begin{tabular}{c|cccccccccccc}
\hline
&$\eta$& $\tau$&$\gamma$& $Q_h$ & $V_h$ & $\check{Q}_h$& $\check{V}_h$ 
\\\hline
1 &$\mathcal{O}(h_e)$& $\mathcal{O}(h_e^{-1})$&$\mathcal{O}(1)$&$ Q_h^{k}$ &  $V_h^{k+1}$& $\check{Q}_h^{r}$  &$\check{V}_h^{k}$
%\\
%2&$\tau^{-1}$& $\Omega(h_e^{-1})$&0&$ Q_h^{k}$ &$V_h^{k+1}$ & $\check{Q}_h^{k}$  &$\check{V}_h^{k}$
\\\hline
2&$\mathcal{O}(h_e^{-1})$&$\mathcal{O}(h_e)$& $\mathcal{O}(1)$ &$Q_h^{k+1}$ & $V_h^{k}$ & $\{0\}\ \text{or}\ \check{Q}_h^{m}$  &$\check{V}_h^{k+1}$ 
\\
3&$\tau^{-1}$&$\Omega(h_e)$&0&$Q_h^{k+1}$&$V_h^{k}$ & $\check{Q}_h^{k}$  &$\check{V}_h^{k+1}$ 
\\\hline
\end{tabular}
\caption{\footnotesize{New proposed methods with $r\ge \max (1, k)$ and $m\ge 0$.  For the second and third schemes, 
$\|\csig-\dsig\|_{\rm div,h}=\cO(h^{k+1})$ for any $k\ge 0$ and $\|\csig-\dsig\|_0=\cO(h^{k+2})$ if $k\ge n$.}}
\label{tab:intronew}
\end{table}
%\begin{table}[!htbp]
%\scriptsize
%\centering
%\begin{tabular}{c|cccccccccccc}
%\hline
%&$\eta$& $\tau$&$\gamma$& $Q_h$ & $V_h$ & $\check{Q}_h$& $\check{V}_h$ &$\|\csig-\dsig\|_0$& $\|u-\du\|_0$& $\|\epsilon (u)-\epsilon_h (\du)\|_0$& $\|\divh_h (\csig-\dsig)\|_0$
%\\\hline
%1 &$\mathcal{O}(h_e)$& $\mathcal{O}(h_e^{-1})$&$\mathcal{O}(1)$&$ Q_h^{k}$ &  $V_h^{k+1}$& $\check{Q}_h^{r}$  &$\check{V}_h^{k}$& $ k+1 $&$ k+2 $&$k+1$  & -
%\\
%2&$\tau^{-1}$& $\Omega(h_e^{-1})$&0&$ Q_h^{k}$ &$V_h^{k+1}$ & $\check{Q}_h^{k}$  &$\check{V}_h^{k}$& $ k+1 $&$ k+2 $&$k+1$  & -
%\\\hline
%3&$\mathcal{O}(h_e^{-1})$&$\mathcal{O}(h_e)$& $\mathcal{O}(1)$ &$Q_h^{k+1}$ & $V_h^{k}$ & $\{0\}\ \text{or}\ \check{Q}_h^{m}$  &$\check{V}_h^{k+1}$ & $ k+1 $&$ k+1 $& -  &$k+1$
%\\
%4&$\tau^{-1}$&$\Omega(h_e)$&0&$Q_h^{k+1}$&$V_h^{k}$ & $\check{Q}_h^{k}$  &$\check{V}_h^{k+1}$ &  $ k+1 $&$ k+1 $& -  &$k+1$
%\\\hline
%\end{tabular}
%\caption{\footnotesize New proposed methods with $r\ge \max (1, k)$ and $m\ge 0$. If $k\ge n$, $\|\csig-\dsig\|_0=\cO(h^{k+2})$ for the last two schemes.}
%\label{tab:intronew}
%\end{table}
Throughout this paper, we shall use letter $C$, which is independent
of mesh-size $h$ and stabilization parameters $\eta, \tau, \gamma$, 
to denote a generic
positive constant which may stand for different values at different
occurrences.  The notation $x \lesssim y$ and $x \gtrsim y$ means $x
\leq Cy$  and $x \geq Cy$, respectively. Denote $x\lesssim y\lesssim x$ by $x \eqsim y$.

The rest of the paper is organized as follows. Some notation is introduced in Section \ref{notation}. 
{In Section \ref{sec:4form}, a four-field unified formulation is derived for linear elasticity problems.  By proving uniform inf-sup conditions under two sets of assumptions, an optimal error analysis is provided for this unified formulation. 
Section \ref{sec:variants} derives some variants of this four-field formulation, and reveals their relation with some existing methods in the literature. 
Section \ref{sec:limit} illustrates two limiting cases of the unified formulation: mixed methods and primal methods. 
Numerical results are provided in Section \ref{sec:numerical} to verify the theoretical analysis including the optimal convergence of the new proposed methods.
Some conclusion remarks are given in Section \ref{concl}. 
}

\section{Preliminaries}\label{notation}

Given a nonnegative integer $m$ and a bounded domain $D\subset \mathbb{R}^n$, let $H^m(D)$, $\|\cdot\|_{m,D}$ and $|\cdot|_{m,D}$ be the usual Sobolev space, norm
and semi-norm,  respectively.  The
$L^2$-inner product on $D$ and $\partial D$ are denoted by $(\cdot,
\cdot)_{D}$ and $\langle\cdot, \cdot\rangle_{\partial D}$,
respectively. Let  $\|\cdot\|_{0,D}$ and $\|\cdot\|_{0,\partial D}$ be the norms of Lebesgue spaces $L^2(D)$ and $L^2(\partial D)$, respectively. The norms $\|\cdot\|_{m,D}$ and $|\cdot|_{m,D}$ are abbreviated as  $\|\cdot\|_{m}$ and
$|\cdot|_{m}$, respectively, when $D$ is chosen as $\Omega$.

Suppose that $\Om\subset \mathbb{R}^n$ is a bounded polygonal domain covered exactly by a shape-regular partition $\cTh$ { of polyhedra}.  Let  $h_K$ be the diameter of element $K\in \cTh$ and $h=\max_{K\in\cTh}h_K$.  Denote the set of all interior edges/faces of $\cTh$ by $\cEi$, and all edges/faces on boundary $\Gamma_D$ and $\Gamma_N$ by $\cED$ and $\cEN$, respectively.  Let $\cE=\cEi\cup \cED \cup \cEN$ and $h_e$ be the diameter of edge/face $e\in \cE$. For any interior edge/face $e=K^+\cap K^-$, let $\bn^i$ = $\bn|_{\partial K^i}$ be the unit outward normal vector on $\partial K^i$ with $i = +,-$.  For any vector-valued function $\dv$ and matrix-valued function $\dtau$, let $\dv^{\pm}$ = $\dv|_{\partial K^{\pm}}$, $\dtau^{\pm}$ = $\dtau|_{\partial K^{\pm}}$.  Define the average $\{\cdot\}$ and the jump $[\cdot ]$ on interior edges/faces $e\in \cEi$ as
follows: 
\begin{equation}\label{jumpdef}
\begin{array}{ll}
\{\dtau\}=\frac{1}{2}(\dtau^++\dtau^-),&[\dtau]=\dtau^+\bn^++\dtau^-\bn^-,\\
\{\dv\}=\frac{1}{2}(\dv^++\dv^-),&[\dv] =\dv^+\odot \bn^++\dv^-\odot \bn^- - (\dv^+\cdot \bn^+ + \dv^-\cdot \bn^-)\bm I
\end{array}
\end{equation} 
where $\dv\odot \bn= \dv\bn^T+ \bn\dv^T${ and $\bm I $ is the identity tensor}. For any boundary edge/face $e\subset \partial \Omega$, define
\begin{equation}\label{bddef}
\begin{array}{lllll}
\{\dtau\}=\dtau, &  [\dtau]=0, &\{\dv\}=\dv,&[\dv]=\dv\odot \bn - (\dv\cdot \bn)\bm I ,& \text{on }\Gamma_D,\\
\{\dtau\}=\dtau,&  [\dtau]=\dtau\bn, &\{\dv\}=\dv, & [\dv]=0,& \text{on }\Gamma_N.
\end{array}
\end{equation}
Note that  the jump $[\dv]$ in \eqref{jumpdef} is a symmetric tensor and 
\begin{equation}\label{vjumpn}
[\dv]\bn^+=\dv^+ - \dv^-,\qquad  \forall e\in \cE.%[\dv]\bn   = \dv,\qquad \forall e\in \cEi,\ e'\subset \Gamma_D.
\end{equation}
These properties are important for the Nitche's technique in \eqref{bdconstrain}, since the trace of the stress tensor $\dsig$ should be a symmetric tensor.
Define some inner products as follows:
\begin{equation} \label{equ:inner-product}
(\cdot,\cdot)_\cTh=\sum_{K\in \cTh }(\cdot,\cdot)_{K},
\quad \langle\cdot,\cdot\rangle =\sum_{e\in \cE}\langle\cdot,\cdot\rangle_{e},
\quad \langle\cdot,\cdot\rangle_{\partial\cTh }=\sum_{K\in\cTh}\langle\cdot,\cdot\rangle_{\partial K}.
\end{equation} 
With the aforementioned definitions, there exists the following identity \cite{arnold2002unified}:
\begin{equation}\label{identities}
\langle \dtau\bn, \dv\rangle_\PcT = \langle \{\dtau\}\bn, [\dv]\bn\rangle + \langle [\dtau], \{\dv\}\rangle.
\end{equation}
For any vector-valued function $\dv$ and matrix-valued function $\dtau$,  define the piecewise gradient $\epsilon_h$ and piecewise divergence ${\rm div}_h$ by
$$
\epsilon_h (\dv)\big |_K=\epsilon (\dv|_K), \quad
\divh_h \dtau\big |_K=\divh (\dtau |_K) \quad \forall K \in \cTh.
$$
Whenever there is no ambiguity, we simplify $(\cdot, \cdot)_\cTh$ as $(\cdot,\cdot)$. 
The following crucial DG identity follows from integration by parts and \eqref{identities}
\begin{equation}\label{DGidentity}
(\dtau, \epsilon_h (\dv))
=-(\divh_h \dtau, \dv)
+ \langle [\dtau], \{\dv\}\rangle
+ \langle \{\dtau\}\bn, [\dv]\bn\rangle.
\end{equation}

\section{A four-field formulation and unified analysis}\label{sec:4form} 
Let $Q_h$ and $V_h$ be approximations to $L^2(\Om, \S)$ and $L^2(\Om, \R^n)$, respectively, and be piecewise smooth  with respect to $\mathcal{T}_h$. Let
$$
\check{Q}_{h}=\{\htau \in L^2(\mathcal E_h, \S): \htau\bn|_{\Gamma_N}=0 \} \quad \text{ and }\quad
\check{V}_{h}=\{\hv \in L^2(\mathcal E_h, \mathbb{R}^n): \hv|_{\Gamma_D}=0 \}.
$$
We start with multiplying the first two equations in \eqref{model} by $\dtau\in Q_h$ and $\dv\in V_h$, respectively. It is easy to obtain that, for any $K\in \mathcal{T}_h$,
\begin{equation}\label{XGelement}
\left\{
\begin{array}{rll}
(A\csig,\dtau)_{0, K}
+(u, \divh_h \dtau)_{0, K}
-\langle u, \dtau\bn \rangle_{0,\partial K}
&=0,&\ \forall \dtau \in Q_h,\\
-(\csig, \epsilon_h (\dv))_{0, K}
+\langle \csig\bn , \dv\rangle_{0, \partial K}
&=(f,\dv)_{0, K}, &\ \forall \dv\in V_h.
\end{array}
\right.
\end{equation}
We introduce two independent discrete variables $\hsig\in \check Q_{h}$ and $\hu\in \check V_{h}$  %here to approximate $\{\csig\}$ and $u$ on edges 
as
\begin{equation}\label{fluxdef}
\csig|_{\partial K}\approx \hat \csig_{h} :=\acute {\csig}_h  + \hsig,
\qquad
u|_{\partial K} \approx  \hat u_{h} :=  \acute{u}_h +\hu,
\end{equation}
where $\acute{\csig}_h=\acute {\csig}_h(\dsig,\du)$ and $\acute {u}_h=\acute {u}_h(\dsig,\du)$ are given in terms of $\dsig$ and $\du$. %, namely $\tilde {\csig}_h=\tilde {\csig}_h(\dsig,\du)$ and $\tilde  {u}_h=\tilde {\csig}_h(\dsig,\du)$. 
Here $\hsig\in \check{Q}_{h}$ and $\hu\in \check{V}_{h}$ are some \textit{residual corrections} to $\acute  {\csig}_h$ and $ \acute  {u}_h$ along interfaces of mesh, respectively. Thus the formulation \eqref{XGelement} can be written as
\begin{equation}\label{elemhat}
\left\{
\begin{array}{rll}
(A\dsig,\dtau)_{0, K}
+(\du, \divh_h \dtau)_{0, K}
-\langle \hat u_{h}, \dtau\bn \rangle_{0, \partial K}
&=0,&\ \forall \dtau \in Q_h,\\
-(\dsig, \epsilon_h (\dv))_{0, K}
+\langle \hat{\csig}_{h}\bn, \dv\rangle_{0, \partial K}
&=(f,\dv)_{0, K}, &\ \forall \dv\in V_h.
\end{array}
\right.
\end{equation}
In order to preserve the continuity of the displacement and { the normal component of }stress across interfaces weakly, we employ  two other equations following the Nitche's technique to determine $\hsig$ and $\hu$
\begin{equation}\label{bdconstrain}
\left\{
\begin{array}{rll}
\langle    \hsig+ \tau[\du], \htau \rangle_{e }
&=0,&\forall  \htau \in \check{Q}_{h},
\\
\langle  \hu + \eta[\dsig], \hv \rangle_e
&=0,&\forall  \hv \in \check{V}_{h}.
\end{array}
\right.
\end{equation}
%The role of the auxiliary parameters $\eta$ and $\tau$ is to enhance the approximate continuity across element boundaries of the discrete strain tensor $\dsig$ and  the displacement $\du$, respectively.
The variable  $\hu$ is not only a residual correction but also a multiplier on the jump $[\dsig]$ along interfaces. Similarly, the variable  $\hsig$ is not only a residual correction but also a multiplier on the jump $[\du]$ along interfaces.
In this paper, we will discuss a special case with %\cite{chen2010local}
\begin{equation}\label{tildedef}
\acute{\csig}_h = \{\dsig\} + [\dsig]\gamma^T,
\qquad
\acute u_h=  \{\du\} - (\gamma^T\bn)[\du]\bn,%- [\du]\gamma,
\end{equation}
where $\gamma\in \R^n$ is a column vector. Thus, 
\begin{equation}\label{hatdef}
\tsig =  \{\dsig\} + [\dsig]\gamma^T +\hsig, \qquad
\tu = \{\du\} - (\gamma^T\bn)[\du]\bn +\hu.
\end{equation} 
\begin{remark}
Note that the formulation, which seeks $(\dsig, \hsig, \du, \hu)\in Q_h\times \check Q_{h}\times V_h\times \check V_{h}$ satisfying \eqref{elemhat}  and \eqref{bdconstrain}, is consistent, since $(\csig, 0, u, 0)$ satisfies the equation \eqref{elemhat}  and \eqref{bdconstrain}  if $(\csig, u)$ is the solution to the model \eqref{model}.
\end{remark}

\subsection{$H^1$-based four-field formulation}
Let $\eta_1=\tau^{-1}$ and $\eta_2=\eta$. By the DG identity \eqref{DGidentity}, the resulting $H^1$-based four-field formulation seeks  $(\dsig, \hsig, \du, \hu)\in Q_h\times \check Q_{h}\times V_h\times \check V_{h}$ such that 
\begin{equation}\label{XGgrad}
\left\{
\begin{array}{rll}
(A\dsig,\dtau)_{0, K}
- (\epsilon_h(\du),  \dtau)_{0, K}
-\langle \hat u_{h} - \du, \dtau\bn \rangle_{0, \partial K}
&=0,&\ \forall \dtau \in Q_h,\\
-(\dsig, \epsilon_h (\dv))_{0, K}
+\langle \hat{\csig}_{h}\bn, \dv\rangle_{0, \partial K}
&=(f,\dv)_{0, K},  &\ \forall \dv\in V_h,\\
\langle %2\tau^{-1}   
\eta_1\hsig+ [\du], \htau \rangle_{e }
&=0,&\forall  \htau \in \check{Q}_{h},
\\
\langle %2\eta^{-1} 
\hu + \eta_2[\dsig], \hv \rangle_e
&=0,&\forall  \hv \in \check{V}_{h},
\end{array}
\right.
\end{equation} 
with $(\tsig, \tu)$ defined in \eqref{hatdef}.
%Dislike the fact that the parameters $\eta$ and $\tau$ in \eqref{bdconstrain} must be nonzero, the parameters $\eta_1$ and $\eta_2$ in \eqref{XGgrad} can be zero.

{
Denote the $L^2$ projection onto $\check{Q}_{h}$ and $\check{V}_{h}$ by $\check P_h^\sigma$ and  $\check P_h^u$, respectively. 
Nitche's technique in \eqref{bdconstrain} implies that
\begin{equation}\label{hatjumprelu}
\hu = -\eta \check P_h^u [\dsig].
\end{equation}
By plugging in the above equation and the identity \eqref{identities} into \eqref{elemhat}, the four-field formulation \eqref{XGgrad} with $(\dsig, \hsig, \du, \hu)$ is equivalent to the following  three-field formulation, which seeks $(\dsig, \hsig, \du)\in Q_h \times \check{Q}_{h}\times V_h $ such that
\begin{equation}\label{M}
\left\{
\begin{array}{rlr}
a_W(\dsig, \hsig; \dtau, \htau) + b_W(\dtau, \htau; \du)
&=0, 
&\forall~(\dtau, \htau)\in Q_h \times\check{Q}_{h},
\\
b_W(\dsig, \hsig; \dv) 
&=(f,\dv),  
&\forall~\dv \in V_h,
\end{array}
\right.
\end{equation}
with
\begin{equation}\label{WGABC}
\left\{
\begin{array}{rl}
a_W(\dsig, \hsig; \dtau, \htau)&=(A\dsig,\dtau)
+\langle \eta_2\check{P}_h^u[\dsig] , [\dtau]\rangle
+ \langle \eta_1   \hsig , \htau \rangle,
\\
b_W(\dsig, \hsig; \dv)&= -(\dsig, \epsilon_h (\dv))
+\langle  (\{\dsig\}  + \hsig + [\dsig]\gamma^T)n, [\dv]\bn\rangle.
\end{array}
\right.
\end{equation}
Thanks to this equivalence, we will use the wellposedness of the three-field formulation \eqref{M} to prove that of
the proposed four-field formulation \eqref{XGgrad} under the following $H^1$-based assumptions: 
}
%For this four-field formulation \eqref{XGgrad}, consider the following $H^1$-based assumptions:
\begin{enumerate}
\item[(G1)] $\epsilon_h (V_h)\subset Q_h$, $\epsilon_h (V_h)|_\cE \subset \check Q_h$ and $Q_h\bn|_\cE\subset \check Q_h$;
\item[(G2)] $\check{Q}_h$ contains piecewise linear functions;
\item[(G3)] $\eta_1 = \rho_1 h_e$, $\eta_2=\rho_2h_e$ and there exist  positive constants $C_1$, $C_2$ and $C_3$ such that
$$
0< \rho_1\leq C_1,\quad  0< \rho_2\leq C_2,\quad 0\leq\gamma\leq C_3,
$$
namely $0<  \eta\leq Ch_e$ and $\tau\ge C h_e^{-1}$ in \eqref{bdconstrain}.
\end{enumerate}
Define
  \begin{equation}\label{H1norms}
\begin{array}{ll}
\|\dtau\|_{0, h}^2 =(A\dtau, \dtau)+\|\eta_1^{1/2}\{\dtau\}\|_\cE^2+ \| \eta_2^{1/2} \check{P}_h^u[\dtau ]\|_\cE^2,
&
\|\htau\|_{0, h}^2 =\|\eta_1^{1/2}\htau\|_\cE^2,
\\
\|\dv\|_{1, h}^2 =\|\epsilon_h (\dv)\|_0^2 + \|\eta_1^{-1/2}\check{P}_h^\sigma[\dv]\|_\cE^2,
&
\|\hv\|_{0, h}^2 =\|\eta_2^{-1/2}\hv\|_\cE^2.
\end{array}
\end{equation}
Assumption (G2) guarantees that  $\|\dv\|_{1, h}$ is a norm for $V_h$. %By the assumption (G3),  trace inequality and inverse inequality,  $\|\dtau\|_{0, h}^2$ is equivalent to $(A\dtau, \dtau)+(\dtau, \dtau)$.  
It follows from \eqref{jumpdef}  that
$$
[\dv] =(\dv^+ - \dv^-)\odot \bn^+ - (\dv^+ - \dv^-)\cdot \bn^+ \bm I.
$$
Thus, by \eqref{vjumpn},
\begin{equation}
\|[\dv]\|_{0, e}\leq 2\|\dv^+ - \dv^-\|_{0, e} = 2\|[\dv] \bn^+\|_{0, e}.
\end{equation}
%$$
%\check P_h^\sigma [\du] =  [\check P_h^\sigma \du]
%$$
This implies that the norm $\|\eta_1^{-1/2} \check P_h^\sigma [\du]\|_{\cE} $ is equivalent to $\|\eta_1^{-1/2} \check P_h^\sigma [\du]\bn\|_{\cE} $, namely,
\begin{equation}\label{normaljump}
c_1\|\eta_1^{-1/2} \check P_h^\sigma [\du]\|_{\cE}\leq \|\eta_1^{-1/2} \check P_h^\sigma [\du]\bn\|_{\cE} \leq c_2\|\eta_1^{-1/2} \check P_h^\sigma [\du]\|_{\cE}.
\end{equation} 
%by
%$$
%\langle \check P_h^\sigma \check {\boldsymbol \sigma}, \htau \rangle = \langle \check{\boldsymbol  \sigma}, \htau\rangle, \quad \forall \htau\in \check{Q}_{h},
%$$ 
%\begin{theorem}\label{lm:equiv}
%The wellposedness of the four-field formulation \eqref{XG} is equivalent to the wellposedness of the two-field formulation \eqref{LDGXG}.
%\end{theorem}
%\begin{proof}
%If the four-field formulation \eqref{XG} is well posed, and $(\dsig, \hsig, \du, \hu)\in Q_h\times \check Q_{h}\times V_h\times \check V_{h}$ is the solution of \eqref{XG}, then $(\dsig, \du)\in Q_h\times V_h$ is the solution of \eqref{LDGXG}. If $(\dsig, \du)\in Q_h\times V_h$ is the solution of \eqref{LDGXG}, define $\hu$ and $\hsig$ by
%$$
% \hu= {\eta \over 2} \check P_h^\sigma [\dsig], \qquad 
%\hsig = {\tau\over 2} \check P_h^u[\du].
% $$
% It is trivial to verify that $(\dsig, \hsig, \du, \hu)\in Q_h\times \check Q_{h}\times V_h\times \check V_{h}$ is the solution of \eqref{XG}, which completes the proof.
%\end{proof} 
Define the lifting operators $r_Q: L^2(\cE, \S)\rightarrow Q_h$ and  $l_Q: L^2(\cE, \mathbb{R}^n)\rightarrow Q_h$ by
\begin{equation}\label{avgQ}
(r_Q(\boldsymbol{\xi}), \dtau)= - \langle  \{\dtau\}\bn, \boldsymbol{\xi}\bn\rangle,\quad (l_Q(w), \dtau)= - \langle  [\dtau], w\rangle,\quad \forall\dtau\in Q_h,  
\end{equation}
respectively, and define $r_V: L^2(\cE, \mathbb{R}^n)\rightarrow V_h$ and  $l_V: L^2(\cE, \S)\rightarrow V_h$ by
\begin{equation}\label{avgV}
(r_V(w), \dv)=  -\langle  \{\dv\}, w\rangle,\quad (l_V(\boldsymbol{\xi}), \dv)= - \langle  [\dv]\bn, \boldsymbol{\xi}\rangle,\qquad \forall\dv\in V_h,
\end{equation}
respectively.
If $w|_e\in P_k(e, \R^n)$, there exist the following estimates \cite{arnold2002unified}
\begin{equation}\label{lift} 
 \|r_Q(\boldsymbol{\xi})\|_0^2\eqsim \|l_V(\boldsymbol{\xi})\|_0^2 \eqsim \|h_e^{-1/2}\boldsymbol{\xi}\|_{\cE}^2,\quad
\|l_Q(w)\|_0^2\eqsim\|r_V(w)\|_0^2 \eqsim  \|h_e^{-1/2}w\|_{\cE}^2. 
\end{equation}

\begin{theorem}\label{Th:inf-supGrad}
Under Assumptions (G1)--(G3), the formulation \eqref{XGgrad} is uniformly well-posed with respect to the mesh size,  $\rho_1$ and $\rho_2$.  Furthermore, there exist the following properties:
\begin{enumerate}
\item  Let $(\dsig, \hsig, \du, \hu)\in Q_h\times \check Q_{h}\times V_h\times \check V_{h}$ be the solution of \eqref{XGgrad}. There exists
\begin{equation}\label{stability}
\|\dsig\|_{0, h}+ \|\hsig \|_{0, h} + \|\du\|_{1, h} +\|\hu\|_{0, h}\lesssim \|f\|_{-1,h}
\end{equation}
with $\|f\|_{-1, h}= \sup\limits_{v_h\in V_h \setminus\{0\}} \frac{(f,
  v_h)}{\|v_h\|_{1,h}}$.
%{ When $\eta_1=0$, the norm $\|\dv\|_{0,h}=\|\epsilon_h(\dv)\|$ without the second term and $\du$ is continuous in some sense. $\|\htau\|_{0,h}=\|\rho_1h_e^{1/2} \htau\|_0$}
\item Let $(\csig, u)\in H^{\frac{1}{2}+\epsilon}(\Om, \S)\cap H(\divh, \Om, \S)\times H^1(\Om, \R^n)$ be the solution of \eqref{model} and $(\dsig, \hsig, \du, \hu)\in Q_h \times \check Q_{h} \times V_h \times \check V_{h} $ be the solution of the formulation \eqref{XGgrad}, the quasi-optimal approximation holds as follows:
\begin{equation}\label{optimal_error}  
\begin{aligned}
&\|\csig-\dsig\|_{0, h}+ \|\hsig \|_{0, h} + \|u-\du\|_{1, h} +\|\hu\|_{0, h}
\\
\lesssim &\inf_{\dtau\in Q_h, \dv\in V_h} \big ( \|\csig-\dtau\|_{0, h}  + \|u-\dv\|_{1, h}\big ). 
\end{aligned}
\end{equation}
%with constant $C_g$ independent  of $h$.
\item If $\csig\in H^{k+1}(\Om, \S)$, $u\in H^{k+2}(\Om, \R^n) ( k\ge 0 )$ and let $(\dsig, \hsig, \du, \hu)\in Q_h^k\times \check Q_{h}^r\times V_h^{k+1}\times \check V_{h}^k$ be the solution of \eqref{XGgrad} with $r{\ge }\max(1, k)$, then we have the following error estimate:
\begin{equation}\label{error}
\|\csig-\dsig\|_{0, h}+ \|\hsig \|_{0, h} + \|u-\du\|_{1, h} +\|\hu\|_{0, h}\lesssim h^{k+1}(|\csig|_{k+1} + |u|_{k+2}).
\end{equation}
%with a constant $C_g$ independent  of $h$.
\end{enumerate}
\end{theorem} 
%\begin{theorem}\label{lm:Grad}
%Under the Assumptions (G1)--(G3), the formulations \eqref{XGgrad}, \eqref{XGdiv}, \eqref{XGH},\eqref{XGW} and \eqref{XGDG} are uniformly well posed with respect to 
%mesh size, $\rho_1, \rho_2$.
%\end{theorem}
\begin{proof}
{
Since the four-field formulation \eqref{XGgrad} is equivalent to the three-field formulation \eqref{M}, it  suffices to prove that \eqref{M} is well-posed under Assumptions (G1) -- (G3), namely the coercivity of $a_W(\cdot, \cdot ; \cdot, \cdot)$ and inf-sup condition for $b_W(\cdot, \cdot ; \cdot)$ in \eqref{WGABC}. 

By the definitions of bilinear form $a_W(\cdot, \cdot ; \cdot, \cdot)$ and norms in \eqref{H1norms}, 
\begin{equation}\label{three-aco}
a_W(\dtau, \htau; \dtau, \htau)\ge c\left(\|\dtau\|_{0, h}^2+\|\htau\|^2_{0, h}\right),\quad \forall \dtau\in Q_h, \htau \in \check Q_h, 
\end{equation}
which is coercive on $Q_h\times \check Q_h$. 

For  any $\dv\in V_h$, take $\dtau=\epsilon_h(\dv)\in Q_h$ and $\htau=\eta_1^{-1} \check{P}_h^\sigma[\dv] + \{\epsilon_h(\dv)\} +  [\epsilon_h(\dv)\gamma^T]$. It holds that
\begin{equation} \label{eq:grad:1}
\begin{split}
b_W(\dtau, \htau; \dv)& = (\epsilon_h (\dv), \epsilon_h (\dv)) + \langle \eta_1^{-1} \check{P}_h^\sigma[\dv]\bn,  \check{P}_h^\sigma[\dv]\bn\rangle
\gtrsim \|\dv\|_{1, h}^2.
\end{split}
\end{equation}
By trace inequality and inverse inequality, we have 
\begin{equation}\label{eq:grad:3}
\begin{split}
\|\dtau\|_{0, h}^2 + \|\htau\|_{0, h}^2
=&(A\epsilon_h (\dv), \epsilon_h (\dv))+\|\eta_1^{1/2}\{\epsilon_h (\dv)\}\|_0^2+ \| \eta_2^{1/2} \check{P}_h^u[\epsilon_h (\dv)]\|_0^2
\\
&+\|\eta_1^{1/2}(\eta_1^{-1} \check{P}_h^\sigma[\dv] + \{\epsilon_h(\dv)\} + [\epsilon_h(\dv)\gamma^T])\|_0^2
\\
\lesssim& \|\epsilon_h (\dv)\|_0^2  + \|\eta_1^{-1/2}\check{P}_h^\sigma[\dv]\|_0^2=\|\dv\|_{1, h}^2. %C(1+\rho_1 + \rho_2)
\end{split}
\end{equation}
It follows that 
\begin{equation}\label{three-binfsup}
\inf_{v_h\in V_h}\sup_{(\dtau,\htau)\in Q_h\times \check Q_h} \frac{b_W(\dtau, \htau; \dv)}{(\|\dtau\|_{0, h}+\|\htau\|_{0, h}) \|\dv\|_{1, h}}\gtrsim 1. 
\end{equation}
By Theorem 4.3.1 in \cite{boffi2013mixed}, a combination of  \eqref{three-aco} and \eqref{three-binfsup} completes the proof.
}
\end{proof}

\begin{remark}\label{remarkgrad}
For the case $\eta_1=0$, the third equation in \eqref{XGgrad} implies that $\check P^\sigma_h [\du]=0$. The corresponding discrete space for $\du$ becomes
$$
V_h^P = \{\dv\in V_h:  \langle  [\dv], \htau \rangle_{e }=0,\ \forall \htau\in \check Q_h\},
$$
and the norm for $\du$ reduces to
$$
\|\du\|_{1, h} =\|\epsilon_h (\du)\|_0.
$$
For this case, $\dsig$, $\du$ and $\hu$ are unique for the four-field formulation \eqref{XGgrad}. The error estimates  \eqref{stability}, \eqref{optimal_error} and \eqref{error} in Theorem \ref{Th:inf-supGrad} also hold for this case.

For the case $\eta_2=0$, the last equation in \eqref{XGgrad} implies that $\hu=0$, therefore $\|\hu\|_{0, h}=0$. The error estimates in Theorem \ref{Th:inf-supGrad} still holds for this case.
\end{remark}

%\begin{remark}
%Theorem \ref{Th:inf-supGrad} leads to the error analysis of $H^1$-based four-field formulation \eqref{XGgrad} and the $H^1$-based three-field formulation \eqref{XGH} in Theorem  \ref{Th:H} and  \eqref{XGW} in Theorem  \ref{Th:W}, respectively. 
%\end{remark}

\subsection{$H({\rm div})$-based four-field formulation}
Let $\tau_1=\tau$ and $\tau_2=\eta^{-1}$. Similarly, by applying  the DG identity \eqref{DGidentity} to the second equation in \eqref{elemhat}, the four-field formulation seeks  $(\dsig, \hsig, \du, \hu)\in Q_h\times \check Q_{h}\times V_h\times \check V_{h}$ such that 
\begin{equation}\label{XGdiv}
\left\{
\begin{array}{rll}
(A\dsig,\dtau)_{0, K}
+(\du, \divh_h \dtau)_{0, K}
-\langle \hat u_{h}, \dtau\bn \rangle_{0, \partial K}
&=0,&\ \forall \dtau \in Q_h,\\
(\divh_h \dsig, \dv)_{0, K}
+\langle \hat{\csig}_{h}\bn - \dsig\bn, \dv\rangle_{0, \partial K}
&=(f,\dv)_{0, K},  &\ \forall \dv\in V_h,\\
\langle %2\tau^{-1}   
\hsig+ \tau_1[\du], \htau \rangle_{e }
&=0,&\forall  \htau \in \check{Q}_{h},
\\
\langle %2\eta^{-1} 
\tau_2\hu + [\dsig], \hv \rangle_e
&=0,&\forall  \hv \in \check{V}_{h},
\end{array}
\right.
\end{equation} 
with $(\tsig, \tu)$ defined in \eqref{hatdef}.

{ %Denote the $L^2$ projection onto $\check{Q}_{h}$ and $\check{V}_{h}$ by $\check P_h^\sigma$ and  $\check P_h^u$, respectively. 
Nitche's technique in \eqref{bdconstrain} implies that
\begin{equation}\label{hatjumprel}
\hsig = -\tau \check P_h^\sigma[\du],\quad \hu = -\eta \check P_h^u [\dsig].
\end{equation}
By plugging in the above equations and the identity \eqref{identities} into \eqref{elemhat}, the four-field formulation \eqref{XGdiv} with $(\dsig, \hsig, \du, \hu)$ is equivalent to the following  two-field formulation,  which seeks $(\dsig, \du)\in Q_h \times V_h$ such that
\begin{equation}\label{LDGXG}
\left\{
\begin{array}{rlr}
a_D(\dsig,\dtau) + b_D(\dtau, \du)
&=0,
&\forall \dtau\in Q_h,
\\
b_D(\dsig, \dv) - c_D(\du, \dv) 
&=(f,\dv),
&\forall \dv \in V_h, 
\end{array}
\right.
\end{equation}
with 
\begin{equation}
\left\{
\begin{array}{rl}
a_D(\dsig,\dtau)&=(A\dsig,\dtau)
+\langle   \eta \check{P}_h^u [\dsig], [\dtau]\rangle,
\\
b_D(\dsig, \dv)&=(\divh_h\dsig,\dv)
-\langle  [\dsig] , \{\dv\}\rangle
+\langle (\gamma^T\bn) [\dsig], [\dv]\bn\rangle\\
&=-(\dsig, \epsilon_h (\dv))
+\langle  \{\dsig\}\bn , [\dv]\bn\rangle
+ \langle (\gamma^T\bn) [\dsig], [\dv]\bn\rangle,
\\
c_D(\du,\dv) &=\langle   \tau \check{P}_h^\sigma[\du]\bn, [\dv]\bn\rangle. 
\end{array}
\right.
\end{equation}  
Thanks to this equivalence, we will use the wellposedness of this two-field formulation \eqref{LDGXG} to prove that of
the proposed four-field formulation \eqref{XGdiv} under the following $H(\rm div)$-based assumptions:
}
%Dislike the fact that the parameters $\eta$ and $\tau$ in \eqref{bdconstrain} must be nonzero, the parameters $\tau_1$ and $\tau_2$ in \eqref{XGgrad} can be zero.
%For this four-field formulation \eqref{XGdiv}, consider the following $H(\rm div)$-based assumptions: 
\begin{enumerate}
\item[(D1)] $Q_h=Q_{h}^{k+1}$, $\divh_h Q_h=V_h\subset V_h^k$, $k\ge 0$;
\item[(D2)] $\check V_h^{k+1}\subset \check V_h$; 
\item[(D3)] $\tau_1=\rho_1 h_e $, $\tau_2=\rho_2 h_e$ and there exist positive constants $C_1$, $C_2$,  $C_3$ and $C_4$ such that
$$
{C_1\leq \rho_1\leq C_2,\quad 0< \rho_2\leq C_3},\quad 0\leq \gamma\leq C_4,
$$ 
namely $\eta\ge Ch_e^{-1}$ and { $C_1h_e \leq\tau\leq C_2h_e$}.
\end{enumerate}
%The case that $\tau_1=\tau_2=0$ in Assumption (D3) implies that the parameter $\tau$ in \eqref{bdconstrain}  can be zero and $\eta$ can tend to infinity. Thus, the resulting four-field formulation tends to a mixed method as analyzed in Section \ref{sec:limit}. 
We first state a crucial estimate \cite{wu2017interior} for the analysis of $H({\rm div})$-based  formulation as follows.
\begin{lemma}\label{lm:div}
For any $\du \in V_h^k$, there exists $\drh \in Q_{h}^{k+1}$ such that
\begin{equation}
\divh_h   \drh=\du, \qquad
\|\drh\|_0 + \|\divh_h \drh\|_0 + \| h_e^{-1/2} [\drh ]\|_0\leq C_0 \|\du\|_0.
\end{equation}
and
\begin{equation}
 \langle [\drh], \hv \rangle=0,~\forall~ \hv \in \check V^k_{h}.
\end{equation}
\end{lemma}  
Define
\begin{equation}\label{divnormdef}
\begin{array}{ll}
{
\|\dtau\|_{\rm div, h}^2 =\|\dtau\|_0^2+ \|\divh_h \dtau\|_0^2+ \| \tau_2^{-1/2}[\dtau ]\|_\cE^2,}
&
\|\htau\|_{0, h}^2 =\|\tau_1^{-1/2}\htau\|_\cE^2,\\
\|\dv\|_{0, h}^2 =\|\dv\|_0^2+ \|\tau_1^{1/2}[\dv]\|_\cE^2+\|\tau_2^{1/2}\{\dv\}\|_\cE^2,
&
\|\hv\|_{0, h}^2 =\|\tau_2^{1/2}\hv\|_\cE^2.
\end{array}
\end{equation} 
{ 
A similar result to Lemma 3.3 in \cite{gatica2015analysis} is proved below.
\begin{lemma}\label{L2:eq}
There exists a constant $C > 0$, independent of mesh size $h$, such that 
\begin{equation}
(\dtau,\dtau)\le C \left((A\dtau,\dtau)+ \|\divh_h \dtau\|_0^2+ \| \tau_2^{-1/2}[\dtau ]\|_\cE^2\right),\quad \forall \dtau\in Q_h.
\end{equation}
\end{lemma}
\begin{proof}
Denote $A_\infty\dtau={1+\nu\over E}\left(\dtau -{1\over n} tr(\dtau)I\right)$ and $c_\nu= {1+\nu\over E}\cdot\frac{1-2\nu}{n+n(n-2)\nu}>0$.
It is obvious that 
\begin{equation}\label{Ainfty}
(A\dtau,\dtau)=(A_\infty \dtau+c_\nu tr(\dtau)I , \dtau) =(A_\infty\dtau, \dtau) +c_\nu\|tr(\dtau)\|_0^2> (A_\infty\dtau, \dtau).
\end{equation}
Following the proof of Lemma 3.3 in \cite{gatica2015analysis}, there exists a positive constant $C$ such that 
\begin{equation}\label{L2Stokes}
(\dtau,\dtau)\le C \left((A_\infty\dtau,\dtau)+ \|\divh_h \dtau\|_0^2+ \| \tau_2^{-1/2}[\dtau ]\|_\cE^2\right), \quad \forall\dtau\in Q_h,
\end{equation}
where $C$ is independent of mesh size $h$.
Combining \eqref{Ainfty} and \eqref{L2Stokes}, we obtain the desired result. 
\end{proof}
}

\begin{theorem}\label{Th:inf-supDiv}
Under Assumptions (D1)--(D3), the $H(\rm div)$-based formulation \eqref{XGdiv} is well-posed with respect to the mesh size,  $\rho_1$ and $\rho_2$. Furthermore, there exist the following properties:
\begin{enumerate}
\item Let $(\dsig, \hsig, \du, \hu)\in Q_h\times \check Q_{h}\times V_h\times \check V_{h}$ be the solution of \eqref{XGdiv}. There exists
\begin{equation}\label{div_stability}
\|\dsig\|_{\rm div, h}+\|\hsig\|_{0, h} + \|\du\|_{0, h} +\|\hu\|_{0, h}\lesssim \|f\|_0.
\end{equation} 
\item Let $(\csig, u)\in H^{\frac{1}{2}+\epsilon}(\Om, \S)\cap H(\divh, \Om, \S)\times H^1(\Om, \R^n)$ be the solution of \eqref{model} and $(\dsig, \hsig, \du, \hu)\in Q_h \times \check Q_{h} \times V_h \times \check V_{h} $ be the solution of the formulation \eqref{XGdiv}, the quasi-optimal approximation holds as follows:
\begin{equation}\label{div_optimal_error}
\begin{aligned}
&\|\csig-\dsig\|_{\rm div, h}+\|\hsig\|_{0, h} + \|u-\du\|_{0, h} +\|\hu\|_{0, h}
\\
\lesssim &\inf_{\dtau\in Q_h,
\dv\in V_h} \big ( \|\csig-\dtau\|_{\rm div, h}  + \|u-\dv\|_{0, h}\big ).
\end{aligned}
\end{equation}
\item
If $\csig\in H^{k+2}(\Om, \S)$, $u\in H^{k+1}(\Om, \R^n) ( k\ge0 )$ and let $(\dsig, \hsig, \du, \hu)\in Q_h^{k+1}\times \check Q_{h}^k\times V_h^{k}\times \check V_{h}^{k+1}$ be the solution of \eqref{XGdiv}, then we have the following error estimate:
\begin{equation}\label{div_error}
\|\csig-\dsig\|_{\rm div, h}+\|\hsig\|_{0, h} + \|u-\du\|_{0, h} +\|\hu\|_{0, h}\lesssim h^{k+1}(|\csig|_{k+2} + |u|_{k+1}).
\end{equation}
\end{enumerate}
\end{theorem}
\begin{proof}
Since the four-field formulation \eqref{XGdiv} is equivalent to the two-field formulation \eqref{LDGXG}, it suffices to prove that \eqref{LDGXG} is well-posed under Assumptions (D1) -- (D3).  

Consider the inf-sup of $b_D(\dsig, \dv)=(\divh_h\dsig,\dv) 
-\langle  [\dsig] , \{\dv\}\rangle 
+\langle (\gamma^T\bn) [\dsig], [\dv]\bn\rangle$. According to Lemma \ref{lm:div}, for any $\du\in V_h$, there exists $\dsig\in Q_h$ such that 
$$ 
\divh_h   \dsig=\du, \qquad 
 \langle [\dsig], \{\du\} \rangle_{0, e}= \langle [\dsig], [\du]\bn \rangle_{0, e}=0,
$$
with $\|\dsig\|_0 + \|\divh_h\dsig \|_0 + \|h_e^{-1/2}[\dsig] \|_\cE \lesssim \|\du\|_0 $. Then,
\begin{equation}\label{div1}
b_D(\dsig, \du)=\|\du\|_0^2 \ge c\|\du\|_{0, h}\|\dsig\|_{\rm div, h},%c{1\over \rho_2^{-1/2} + 1}
\end{equation}
which proves the inf-sup condition of $b_D(\cdot, \cdot)$.

Define 
$$  
%N(B)
{ \mathbb K}=\{\dsig\in Q_h: (\divh_h\dsig,\dv) 
-\langle  [\dsig] , \{\dv\}\rangle
+\langle (\gamma^T\bn) [\dsig], [\dv]\bn\rangle
=0,\ \forall \dv\in V_h\}.
$$
It follows from the definition of ${ \mathbb K}$ %$N(B)$ 
and the lifting operator in \eqref{avgV} that 
\begin{equation*}
\divh_h\dsig= -r_V([\dsig]) + l_V((\gamma^T\bn)[\dsig]),\quad \forall \dsig\in { \mathbb K}.%N(B).
\end{equation*}
According to Assumption (D2) { and Lemma \ref{L2:eq},}
\begin{equation}\label{div2}
a_D(\dsig,\dsig)=(A\dsig,\dsig)
+\langle  \tau_2^{-1} [\dsig], [\dsig]\rangle\ge c\|\dsig\|_{\rm div, h}^2.%{c\over \rho_1^{-1}+1}
\end{equation}
This means that $a_D(\cdot, \cdot)$ is coercive on ${ \mathbb K}$. %$N(B)$. 
By Theorem 4.3.1 in \cite{boffi2013mixed}, a  combination of \eqref{div1} and \eqref{div2} leads to the wellposedness of the two-field formulation \eqref{LDGXG}, and completes the proof.
\end{proof}
{
\begin{remark}
Note that the norm $\|\cdot\|_{\rm div, h}$ defined in \eqref{divnormdef} and the constants in \eqref{div1} and \eqref{div2} do not depend on 
the Poisson's ratio $\nu$. 
Hence by Theorem \ref{Th:inf-supDiv}, the proposed formulation \eqref{XGdiv} under Assumptions (D1)--(D3) is locking-free. 
\end{remark}
}
\begin{remark}\label{remarkdiv}
For the case $\tau_1=0$, the third equation in \eqref{XGdiv} implies that $\hsig=0$, therefore $\|\hsig\|_{0, h}=0$. The error estimates in Theorem \ref{Th:inf-supDiv} still holds for this case.

For the case $\tau_2=0$, the last equation in \eqref{XGdiv} implies that $\check P^u_h [\dsig]=0$. The corresponding discrete space for $\dsig$ becomes
$$
Q_h^M = \{\dtau\in Q_h:  \langle  [\dtau], \hv \rangle_{e }=0,\ \forall \hv\in \check V_h\},
$$
and the norm for $\dtau$ reduces to
{
$$
\|\dtau\|_{\rm div, h}^2 =\|\dtau\|_0^2
+ \|\divh_h \dtau\|_0^2.
$$
}
For this case, $\dsig$, $\du$ and $\hsig$ are unique for the four-field formulation \eqref{XGdiv}. The error estimates  \eqref{div_stability}, \eqref{div_optimal_error} and \eqref{div_error} in Theorem \ref{Th:inf-supDiv} also hold for this case.
\end{remark}
%\begin{remark}
%Theorem \ref{Th:inf-supDiv} leads to the error analysis of $H({\rm div})$-based four-field formulation \eqref{XGdiv} and the $H({\rm div})$-based three-field formulation \eqref{XGH} in Theorem  \ref{Th:inf-supDiv} and  Theorem \ref{Th:H}, respectively.  
%\end{remark} 

Let $\mathbb{M}$ be the space of real matrices of size $n\times n$. Given $\bm \sigma_h$ and $\hat {\bm \sigma}_h$, define a matrix-valued function $\ttsig\in \mathcal{P}_{k+1}(K ; \mathbb{M})$:
\begin{equation}\label{def:PiSig}
\begin{array}{ll}
\int_{e}\left(\ttsig -\tsig\right) n \cdot p_{k+1} ds=0, & \forall p_{k+1} \in \mathcal{P}_{k+1}\left(e ; \mathbb{R}^{n}\right), \\
\int_{K}\left(\ttsig -\dsig\right): \nabla p_{k} dx=0, & \forall p_{k} \in \mathcal{P}_{k}\left(K ; \mathbb{R}^{n}\right), \\
\int_{K}\left(\ttsig -\dsig \right): \boldsymbol{p}_{k+1} dx=0, & \forall \boldsymbol{p}_{k+1} \in  \Phi_{k+1}(K),
\end{array}
\end{equation}
where
$
\Phi_{k+1}(K)=\left\{\dtau \in \mathcal{P}_{k+1}(K ; \mathbb{M}): \operatorname{div} \dtau=0,\left.\dtau \bn\right|_{\partial K}=0\right\}.
$
%Here, the $\nabla$ is regarded as the row-wise operator, i.e.,
%\[
%\nabla p=\left(\begin{array}{c}
%\left(\nabla p_{1}\right)^{t} \\
%\vdots \\
%\left(\nabla p_{d}\right)^{t}
%\end{array}\right), \quad p=\left(p_{1}, \cdots, p_{d}\right)^{t}.
%\]

Define the following space
\[
\mathrm{BDM}_{k+1}^{n \times n}:=\left\{\boldsymbol{\tau} \in H(\operatorname{div}, \Omega ; \mathbb{M}):\left.\boldsymbol{\tau}\right|_{K} \in \mathcal{P}_{k+1}(K ; \mathbb{M}),\ \forall K \in \mathcal{T}_{h}\right\},
\]
and the norm
$$
\|\dtau\|_A^2 = (A\dtau, \dtau),\quad \forall \dtau \in L^2(\Omega, \S).
$$
There exists the following estimate in \cite{wang2020mixed}.
\begin{lemma}\label{lm:l2}
The matrix-valued function $\ttsig \in \mathrm{BDM}_{k+1}^{n \times n}$ in \eqref{def:PiSig} is well defined and 
\begin{equation}
\left\|\ttsig -\dsig \right\|_{0, K} \lesssim h_{K}^{1 / 2} \| (\tsig -\dsig)\bn\|_{\partial K}.
\end{equation}
Furthermore, there exists a matrix-valued function $\tttau \in \mathrm{BDM}_{k+1}^{n \times n}$ such that $\dsig^\ast:= \ttsig +\tttau \in H(\operatorname{div}, \Omega, \S),$ and
\[
\operatorname{div} \tttau=0 \text { and }\left\|\tttau\right\|_{0} \lesssim\left\|\dsig -\ttsig\right\|_{0}.
\]
\end{lemma}
Similar to the analysis in \cite{wang2020mixed}, there exists the following $L^2$ error estimate of the discrete stress tensor for the XG formulation. 
\begin{theorem}\label{Th:sigL2div}
Let $\csig\in H^{k+2}(\Om, \S)$ and $u\in H^{k+1}(\Om, \R^n) ( k\ge n)$ be the solution of \eqref{model} and $(\dsig, \hsig, \du, \hu)\in Q_h^{k+1}\times \check Q_{h}^k\times V_h^{k}\times \check V_{h}^{k+1}$ be the solution of \eqref{XGdiv}. Under Assumptions (D1)--(D3), it holds that
\begin{equation}
\|\csig - \dsig\|_A\le h^{k+2}(|\csig|_{k+2} + |u|_{k+1}).
\end{equation} 
\end{theorem}
\begin{proof}
Recall the following $H({\rm div})$ four-field formulation \eqref{elemhat} and \eqref{bdconstrain}
\begin{equation} 
\left\{
\begin{array}{rll}
(A\dsig,\dtau)_{0, K}
+(\du, \divh_h \dtau)_{0, K}
-\langle \hat u_{h}, \dtau\bn \rangle_{0, \partial K}
&=0,&\ \forall \dtau \in Q_h,\\
-(\dsig, \epsilon_h (\dv))_{0, K}
+\langle \hat{\csig}_{h}\bn, \dv\rangle_{0, \partial K}
&=(f,\dv)_{0, K},&\ \forall \dv\in V_h,\\
\langle    \hsig+ \tau[\du], \htau \rangle_{e }
&=0,&\forall  \htau \in \check{Q}_{h},
\\
\langle  \hu + \eta[\dsig], \hv \rangle_e
&=0,&\forall  \hv \in \check{V}_{h}.
\end{array}
\right.
\end{equation}
with $\tsig =  \{\dsig\} + [\dsig]\gamma^T +\hsig$ and $\tu = \{\du\} - (\gamma^T\bn)[\du]\bn +\hu.$ By the second equation in the above equation and the definition of $\ttsig$ in \eqref{def:PiSig},
\begin{equation}\label{divsigstar}
\begin{aligned} 
(f,\dv)=&-(\dsig, \epsilon_h (\dv))+\langle \hat{\csig}_{h}\bn, \dv\rangle_\PcT
=-(\dsig, \nabla_h \dv)+\langle \hat{\csig}_{h}\bn, \dv\rangle_\PcT
\\
=&-(\ttsig, \nabla_h \dv)+\langle \ttsig\bn, \dv\rangle_\PcT
=(\divh \ttsig, \dv)=(\divh \dsig^*, \dv).
\end{aligned} 
\end{equation}
When $k\ge n$, there exists a projection { $\Pi_h^c: H^{1}(\Omega, \S)\rightarrow Q_h\cap H(\rm div,  \Omega, \S)$, see Remark 3.1 in \cite{hu2014finite} for reference, such that
\begin{equation}
\begin{aligned} 
(\divh (\ctau -\Pi_{h}^{c} \ctau), \dv)_{\Omega} &=0 & & \hbox{for any}~ \dv\in V^k_h,
\\
\left\|\ctau - \Pi_{h}^{c} \ctau \right\|_{0, \Omega} & \lesssim h^{k+2}|\ctau |_{k+2, \Omega} & & \hbox{if}~ \ctau \in H^{k+2}(\Omega, \S).
\end{aligned}
\end{equation}
}
It follows from \eqref{divsigstar} and Lemma \ref{lm:l2} that
\begin{align} 
(\divh (\dsig^* -\Pi_{h}^{c} \csig ), \dv)&=(\divh (\ttsig -  \csig ), \dv)+ (\divh \tttau, \dv)=0 
\end{align}
Let $\dtau=\Pi_h^c \csig - \dsig^*\in H(\operatorname{div}, \Omega, \S)$. According to Assumption (D1), $\divh_hQ_h\subset V_h$. Thus, 
$$
\divh \dtau=0.
$$ 
It follows  from {\eqref{hatdef}, \eqref{XGdiv}} and $\dtau\in H(\operatorname{div}, \Omega, \S)$ that
\begin{equation}
\begin{aligned} 
(A(\csig - \dsig), \dtau) =& \langle u - \tu, \dtau\bn\rangle_{\PcT} -  (u - \du, \divh \dtau)=  \langle {u-\tu}, [\dtau]\rangle=0.
\end{aligned}
\end{equation}
Since
\begin{align} 
(A(\csig - \dsig), \csig - \dsig) =&\nonumber(A(\csig - \dsig), \csig - \Pi_h^c \csig) + (A(\csig - \dsig), \dtau) + (A(\csig - \dsig), \dsig^* - \dsig)
\\
=&(A(\csig - \dsig), \csig - \Pi_h^c \csig) +  (A(\csig - \dsig), \dsig^* - \dsig),
\end{align}
we have
\begin{equation}\label{eq:l21}
\begin{aligned}  
\|\csig - \dsig\|_A \le &\|\csig - \Pi_h^c \csig\|_A + \|\dsig^* - \dsig\|_A 
\le \|\csig - \Pi_h^c \csig\|_A + \|\tttau\|_A + \|\ttsig - \dsig\|_A 
\\
\lesssim & \|\csig - \Pi_h^c \csig\|_0 +  \|\dsig -\ttsig\|_0.
\end{aligned}
\end{equation}
A combination of \eqref{hatjumprel} and  Lemma \ref{lm:l2} leads to
\begin{equation}\label{eq:l22}
\begin{aligned}
 \|\dsig -\ttsig\|_0 &\lesssim \| h_{K}^{1 / 2}(\tsig -\dsig) n \|_{\PcT}\lesssim   h^{1 / 2} {(\|\hsig\bn \|_{\cE}  + \| [\dsig]\bn\|_\cE)}
 \\
 &\lesssim h\|\hsig\|_{0, h} +  h\|\hu\|_{0, h}.
 \end{aligned}
\end{equation}
It follows from \eqref{eq:l21} and \eqref{eq:l22} that
$$
\|\csig - \dsig\|_A \lesssim h^{k+2}(|\csig|_{k+2} + |u|_{k+1}),
$$
which completes that proof.
\end{proof}

It needs to point out that the above two {discretizations} \eqref{XGgrad} and \eqref{XGdiv} are mathematically equivalent under the same choice of discrete spaces and parameters. But these two {discretizations} behave differently under different assumptions (G1)--(G3) or (D1)--(D3).  {discretizations} under Assumptions (G1)--(G3) are more alike $H^1$-based {methods} and those under Assumptions (D1)--(D3) are more alike $H({\rm div})$-based {methods}. According to these two sets of assumptions, the parameter $\tau$ in \eqref{bdconstrain}  can tend to infinity in {an} $H^1$-based formulation, but not in {an} $H({\rm div})$-based formulation, while the parameter $\eta$ can tend to infinity in {an} $H({\rm div})$-based formulation, but not in {an} $H^1$-based formulation. In the rest of this paper, we will use \eqref{XGgrad} whenever {an} $H^1$-based formulation is {considered}, and \eqref{XGdiv} for {an} $H({\rm div})$-based formulation.  

\section{Variants of the four-field formulation}\label{sec:variants} 

Note that the last two equations in  \eqref{XGgrad}  and  \eqref{XGdiv} reveal the relations \eqref{hatjumprel} between $\hsig$, $\hu$ and $[\du]$, $[\dsig]$, respectively. In the four-field formulation \eqref{XGgrad} and \eqref{XGdiv}, we can eliminate one or some of the four variables and obtain several reduced formulations as discussed below.

\subsection{Three-field formulation without the variable $\hsig$}
The relations \eqref{hatdef} and \eqref{hatjumprel} imply that 
\begin{equation}\label{SigHDG}
\tsig =  \{\dsig\} + [\dsig]\gamma^T   - \tau_1  \check{P}_h^\sigma [\du].
\end{equation}  
A substitution of \eqref{SigHDG} into the four-field formulation \eqref{XGdiv} gives the following three-field formulation without the variable $\hsig$  which seeks $(\dsig, \du,  \hu)\in Q_h \times V_h \times \check{V}_{h}$ such that
\begin{equation}\label{XGH}
\left\{
\begin{array}{rll}
(A\dsig,\dtau)_{0, K}
+(\du, \divh_h \dtau)_{0, K}
-\langle \hat u_{h}, \dtau\bn \rangle_{0, \partial K}
&=0,&\ \forall \dtau \in Q_h,
\\
-(\dsig, \epsilon_h (\dv))_{0, K}
+\langle \hat{\csig}_{h}\bn, \dv\rangle_{0, \partial K}
&=(f,\dv)_{0, K}, &\ \forall \dv\in V_h,
\\ 
\langle\tau_2  \hu +  [\dsig], \hv \rangle_e
&=0,&\forall  \hv \in \check{V}_{h},
\end{array}
\right.
\end{equation}
with $\tu$ and $\tsig$ defined in \eqref{hatdef} and \eqref{SigHDG}, respectively.  

%Although the parameter $\tau$ in the four-field formulation \eqref{XG} has to be nonzero, $\tau$ in this three-field formulation \eqref{XGH} can be zero if the discrete spaces and parameters satisfying the Assumption (D1), (D2) and (D3') with 
%\begin{enumerate}
%\item[(D3')] $\tau=\rho_1 h_e$, $\eta=\rho_2 h_e^{-1}$ and there exist positive constants $C_1$, $C_2$ and  $C_3$ such that
%$$
%0\leq \rho_1\leq C_1,\quad \rho_2\ge C_2,\quad 0\leq \gamma\leq C_3.
%$$ 
%\end{enumerate}

The equivalence between the four-field formulations \eqref{XGgrad}, \eqref{XGdiv} and the three-field formulation \eqref{XGH} gives the following optimal error estimates.
\begin{theorem}\label{Th:H}
 There exist the following properties:
\begin{enumerate}
\item  Under Assumptions (G1)--(G3), the $H^1$-based formulation \eqref{XGH} is uniformly well-posed with respect to 
mesh size,  $\rho_1$ and $\rho_2$.  Let $(\dsig,\du, \hu)\in Q_h\times V_h\times \check V_{h}$ be the solution of \eqref{XGH}. There exists
\begin{equation} 
\|\dsig\|_{0, h}+ \|\du\|_{1, h} +\|\hu\|_{0, h}\lesssim \|f\|_{-1,h}.
\end{equation}
If $\csig\in H^{k+1}(\Om, \S)$, $u\in H^{k+2}(\Om, \R^n) ( k\ge 0 )$, let $(\dsig,  \du, \hu)\in Q_h^k \times V_h^{k+1}\times \check V_{h}^k$ be the solution of \eqref{XGH}, then we have the following error estimate:
\begin{equation} 
\|\csig-\dsig\|_{0, h} + \|u-\du\|_{1, h} +\|\hu\|_{0, h}\lesssim h^{k+1}(|\csig|_{k+1} + |u|_{k+2}).
\end{equation}
\item Under Assumptions (D1)--(D3), the $H({\rm div})$-based formulation \eqref{XGH} is uniformly well-posed with respect to 
mesh size,  $\rho_1$ and $\rho_2$.  Let $(\dsig, \du, \hu)\in Q_h\times  V_h\times \check V_{h}$ be the solution of \eqref{XGH}. There exists
\begin{equation} 
\|\dsig\|_{\rm div, h} + \|\du\|_{0, h} +\|\hu\|_{0, h}\lesssim \|f\|_0
\end{equation}
If $\csig\in H^{k+2}(\Om, \S)$, $u\in H^{k+1}(\Om, \R^n) ( k\ge0 )$, let $(\dsig, \du, \hu)\in Q_h^{k+1} \times V_h^{k}\times \check V_{h}^{k+1}$ be the solution of \eqref{XGH}, then we have the following error estimate:
\begin{equation}\label{div_error}
\|\csig-\dsig\|_{\rm div, h}  + \|u-\du\|_{0, h} + \|\hu\|_{0, h}\lesssim h^{k+1}(|\csig|_{k+2} + |u|_{k+1}).
\end{equation}
Furthermore, if $k\ge n$,
\begin{equation}\label{div_error2}
\|\csig - \dsig\|_A\lesssim h^{k+2}(|\csig|_{k+2} + |u|_{k+1}).
\end{equation}
\end{enumerate} 
\end{theorem} 

\subsubsection{A special case of the three-field formulation without $\hsig$}
Consider a special case of this three-field formulation \eqref{XGH} with 
\begin{equation}\label{hdgCon}
\tau_2=4\tau_1,\quad \gamma = 0, \quad V_h|_\cE\subset \check V_h\quad V_h|_\cE\subset \check Q_h\bn.
\end{equation} 
It follows from \eqref{SigHDG} that
\begin{equation}\label{errusig}
\langle \tsig \bn , \dv \rangle_\PcT=  \langle\dsig \bn - 2\tau_1 \check P_h^\sigma(\du - \tu ), \dv \rangle_\PcT.
\end{equation}
By eliminating $\tsig$ in \eqref{XGgrad} or \eqref{XGdiv}, we obtain the three-field formulation which seeks $(\dsig, \du,  \hu)\in Q_h \times V_h \times \check{V}_{h}$ such that
\begin{equation} \label{XGHDG0}
\left\{
\begin{array}{rlll}
(A\dsig,\dtau)_{0, K}
+(\du, \divh_h \dtau)_{0, K}
-\langle\tu , \dtau\bn\rangle_{\partial K}
=&0,
&\dtau\in Q_h,
\\
-(\dsig, \epsilon_h(\dv)) _{0, K}
+ \langle  \dsig \bn - 2 \tau_1 \check P_h^\sigma(\du - \tu ), \dv\rangle_{\partial K}
=&(f,\dv),
&\dv\in V_h,
\\
\langle \dsig \bn - 2\tau_1 \check P_h^\sigma(\du - \tu ), \tv \rangle_\PcT
&=0,&\forall  \hv \in \check{V}_{h}.
\end{array}
\right.
\end{equation}  
This reveals the close relation between the three-field formulation \eqref{XGH} and the HDG formulations \cite{fu2015analysis,soon2009hybridizable,chen2016robust,qiu2018hdg}. 
It implies that the special three-field formulation \eqref{XGH} mentioned above is also hybridizable under Assumptions (G1)-(G3). Therefore, the four-field formulation \eqref{XGgrad} or \eqref{XGdiv} with $\tau_2=2\tau_1$ and $\check Q_h\bn=\check V_h$ can be reduced to a one-field formulation with only the variable $\tu$.

Table \ref{tab:HDGexist} lists three HDG methods for linear elasticity problems in the literature and a new  $H({\rm div})$-based method. Since the three-field formulation \eqref{XGHDG0} is equivalent to \eqref{XGgrad} and \eqref{XGdiv}, the new method in Table \ref{tab:HDGexist} is well-posed according to Theorem \ref{Th:inf-supGrad}.
\begin{table}[!htbp] 
\centering
\begin{tabular}{c|ccccccccccccc}
\hline
cases&$\eta$& $\tau$& $\gamma$&$Q_h$ & $\check Q_h$ &$V_h$ & $\check{V}_h$ &
\\ 
\hline
1&$\tau^{-1}$& $\Omega(h_e)$ & 0&$ Q_h^k$ & $\check Q_h^k$&$V_h^{k}$ &
$\check{V}_h^{k}$ &\cite{soon2009hybridizable,fu2015analysis}
\\
2&$\tau^{-1}$& $\Omega(h_e^{-1})$&  0&$ Q_h^k$ & $\check Q_h^k$&$V_h^{k}$ &
$\check{V}_h^k$&\cite{soon2009hybridizable}
\\
3&$\tau^{-1}$& $\Omega(h_e^{-1})$& 0&$ Q_h^{k}$ &$\check Q_h^{k}$ &$V_h^{k+1}$ & $\check{V}_h^{k}$&  \cite{qiu2018hdg,chen2016robust} 
\\\hline
4&$\tau^{-1}$&$\Omega(h_e)$&0&$Q_h^{k+1}$& $\check{Q}_h^{k+1}$  & $V_h^{k}$ &$\check V_h^{k+1}$& new 
\\\hline
\end{tabular}
\caption{Some existing HDG methods and a new HDG method.} 
\label{tab:HDGexist}
\end{table}
%\begin{table}[!htbp]
%\scriptsize
%\centering
%\begin{tabular}{c|ccccccccccccc}
%\hline
%cases&$\eta$& $\tau$& $\gamma$&$Q_h$ & $\check Q_h$ &$V_h$ & $\check{V}_h$ &$\|\csig-\dsig\|_0$  & $\|u-\du\|_0$& $\|\epsilon (u)-\epsilon_h (\du)\|_0$&$\|\divh_h (\csig-\dsig)\|_0$   &
%\\ 
%\hline
%1&$\tau^{-1}$& $\Omega(h_e)$ & 0&$ Q_h^k$ & $\check Q_h^k$&$V_h^{k}$ &
%$\check{V}_h^{k}$ &$k+\frac{1}{2}$ &$k+1$ & $k+\frac{1}{2}$& -&\cite{soon2009hybridizable,fu2015analysis}
%\\
%2&$\tau^{-1}$& $\Omega(h_e^{-1})$&  0&$ Q_h^k$ & $\check Q_h^k$&$V_h^{k}$ &
%$\check{V}_h^k$& $k$& $k+1$ &$k$& -&\cite{soon2009hybridizable}
%\\
%3&$\tau^{-1}$& $\Omega(h_e^{-1})$& 0&$ Q_h^{k}$ &$\check Q_h^{k}$ &$V_h^{k+1}$ & $\check{V}_h^{k}$& $k+1$ & $k+2$&$k+1$  &-&\cite{qiu2018hdg,chen2016robust} 
%\\\hline
%4&$\tau^{-1}$&$\Omega(h_e)$&0&$Q_h^{k+1}$& $\check{Q}_h^{k+1}$  & $V_h^{k}$ &$\check V_h^{k+1}$&$k+1$&$k+1$ &  - &$k+1$& new 
%\\\hline
%\end{tabular}
%\caption{Some existing HDG methods and new HDG methods.} 
%\label{tab:HDGexist}
%\end{table}
\begin{enumerate}
\item The first two HDG methods in this table were proposed in \cite{soon2009hybridizable}, and the first one was then analyzed in \cite{fu2015analysis}. The inf-sup conditions in Theorem \ref{Th:inf-supGrad} and \ref{Th:inf-supDiv} are not optimal for these two cases since the ${\rm degree\ of }\ Q_h$ {equals to the }${\rm degree\ of }\ V_h$.
\item The third one is called the HDG method with reduced stabilization. It was proposed and analyzed { to be a locking-free scheme} in \cite{qiu2018hdg,chen2016robust}. Theorem~\ref{Th:inf-supGrad} provides a brand new proof of the optimal error estimate for this HDG method.
%\item The third one is a standard HDG method proposed in \cite{soon2009hybridizable}. Numerical examples therein showed that the strain tensor converges with order $k$ and the displacement converges with order $k+1$, but there are no theoretical results for the convergence rates.  Theorem \ref{Th:inf-supGrad} offers an optimal error estimate for this standard HDG method.
\item The last one is a new three-field scheme proposed following the $H({\rm div})$-based  formulation \eqref{XGHDG0}. The error estimate for this { locking-free scheme}  is analyzed in Theorem  \ref{Th:H}. Note that the divergence of the stress tensor is approximated by $\divh_h \dsig$ directly in this new $H({\rm div})$-based scheme without any extra post-process as required in $H^1$-based  methods.  
\end{enumerate}
\subsubsection{Hybridization for the $H({\rm div})$-based  formulation \eqref{XGHDG0}}
{ Similar to the hybridization in \cite{qiu2018hdg,chen2016robust}, 
the $H({\rm div})$-based three-field formulation \eqref{XGHDG0} is also hybridizable under Asssumptions (D1)--(D3). 
It can be decomposed into two sub-problems as:}
\begin{enumerate}
\item[(I)] Local problems. For each element $K$, given $\tu \in \check  V_h$, find $(\dsig^K, \du^K)\in Q_h\times V_h$ such that
\begin{equation}\label{XGHDG1}
\left\{
\begin{array}{rll}
(A\dsig^K,\dtau)_K
+(\du^K, \divh \dtau)_K
=&\langle\tu , \dtau\bn\rangle_{\partial K}, 
&\dtau\in Q_h,
\\
(\divh_h\dsig^K, \dv)_K
- \langle 2\tau_1 \du^K, \dv\rangle_{\partial K}
=&(f,\dv)_K
%-\langle (I-\hat P_h^\sigma)g_N, \dv\rangle_{\partial K\cap\Gamma_N}
- \langle 2\tau_1 \tu, \dv\rangle_{\partial K},
&\dv\in V_h.
\end{array}
\right.
\end{equation}
It is easy to see \eqref{XGHDG1} is well-posed. 
Denote $H_Q: \check  V_h\rightarrow Q_h$ and $H_V:\check  V_h\rightarrow V_h$ by 
$$
H_Q(\tu)|_K= \dsig^K\quad\text{ and }\quad  H_V(\tu)|_K= \du^K,
$$
respectively. 
\item[(II)] Global problem. Find $\tu\in \check  V_h$ such that
\begin{equation} \label{XGHDG2}
\langle H_Q(\tu)\bn -  2\tau_1 (H_V(\tu) - \tu), \tv \rangle_\PcT
=0,\quad \tv \in \check  V_h.
\end{equation}
It follows from \eqref{XGHDG1} that
\begin{eqnarray*}
&&(  AH_Q(\tv), H_Q(\tu))_K
+ \langle  H_V(\tv), \divh(H_Q(\tu))\rangle_{\partial K}
=
\langle\tv, H_Q(\tu)\bn \rangle_{\partial K},
\\
&&\langle  2\tau_1(\tu-H_V(\tu)), H_V(\tv)\rangle_{\partial K}
=(f, H_V(\tv))_K- (\divh H_Q(\tu), H_V(\tv))_K.
\end{eqnarray*}
The global problem \eqref{XGHDG2} can be written in the following symmetric positive form
\begin{equation}\label{HDGhybrid}
\begin{split}
(AH_Q(\tu), H_Q(\tv))
+ \langle 2\tau_1 (\tu - H_V(\tu)), \tv-H_V(\tv)\rangle_\PcT
= - (f, H_V(\tv)).
\end{split}
\end{equation}
Since  the original formulation \eqref{XGHDG0} is well-posed, the global problem \eqref{HDGhybrid} is also well-posed. 

\end{enumerate}
 
Suppose Assumptions (D1)--(D3) hold. If the parameter $\tau_1$ is nonzero, the formulation \eqref{XGHDG0} is an $H({\rm div})$-based HDG formulation, and it is hybridizable with only one variable $\tu$ globally coupled in \eqref{HDGhybrid}. 
If the parameter $\tau_1$ vanishes, the formulation \eqref{XGHDG0} is a hybridizable mixed formulation \cite{gopalakrishnan2011symmetric,gong2019new}. 
This implies that the formulation \eqref{XGgrad} or \eqref{XGdiv} with \eqref{hdgCon} can be reduced to a one-field formulation with only the variable $\tu$. 

\subsection{Three-field formulation without the variable $\hu$} 
The relations \eqref{hatdef} and \eqref{hatjumprel} imply that 
\begin{equation}\label{uWG}
\tu = \{\du\} - (\gamma^T\bn)[\du]\bn - \eta_2 \check{P}_h^u[\dsig].
\end{equation}  
Another reduced formulation is resulted from eliminating $\hu$ in the four-field formulation \eqref{XGgrad}  by use of \eqref{uWG}. It seeks $(\dsig, \hsig, \du)\in Q_h \times \check{Q}_{h}\times V_h $ such that
\begin{equation}\label{XGW}
\left\{
\begin{array}{rll}
(A\dsig,\dtau)_{0, K}
-(\epsilon_h(\du), \dtau)_{0, K} 
+\langle \du - \tu, \dtau\bn \rangle_{0, \partial K}
&=0,&\ \forall \dtau \in Q_h,
\\
-(\dsig, \epsilon_h (\dv))_{0, K}
+\langle \hat{\csig}_{h}\bn, \dv\rangle_{0, \partial K}
&=(f,\dv)_{0, K},  &\ \forall \dv\in V_h,
\\ 
\langle \eta_1 \hsig+ [\du], \htau \rangle_{e }
&=0,&\forall  \htau \in \check{Q}_{h},
\end{array}
\right.
\end{equation}
with $\tu$ and $\tsig$ defined in \eqref{uWG} and \eqref{hatdef}, respectively. 
{ The variable  $\hsig$ weakly imposes the $H^1$-continuity of the variable $\du$ in formulation \eqref{XGgrad} or \eqref{XGdiv}. }This makes the three-field formulation \eqref{XGW} more alike primal methods. 
\begin{theorem}\label{Th:W}
 There exist the following properties:
\begin{enumerate}
\item  Under Assumptions (G1)--(G3), the $H^1$-based formulation \eqref{XGW} is uniformly well-posed with respect to 
mesh size,  $\rho_1$ and $\rho_2$.  Let $(\dsig, \hsig, \du)\in Q_h\times \check Q_{h}\times V_h$ be the solution of \eqref{XGW}. There exists
\begin{equation} 
\|\dsig\|_{0, h}+ \|\du\|_{1, h} +\|\hsig\|_{0, h}\lesssim \|f\|_{-1,h}.
\end{equation}
If $\csig\in H^{k+1}(\Om, \S)$, $u\in H^{k+2}(\Om, \R^n) ( k\ge 0 )$, let $(\dsig, \hsig,  \du)\in Q_h^k \times \check Q_{h}^r\times V_h^{k+1}$ be the solution of \eqref{XGW} with $r=\max(1, k)$, then we have the following error estimate:
\begin{equation} 
\|\csig-\dsig\|_{0, h} + \|u-\du\|_{1, h} +\|\hsig\|_{0, h}\lesssim h^{k+1}(|\csig|_{k+1} + |u|_{k+2}).
\end{equation}
\item Under Assumptions (D1)--(D3), the $H({\rm div})$-based formulation \eqref{XGW} is uniformly well-posed with respect to 
mesh size,  $\rho_1$ and $\rho_2$.  Let $(\dsig, \hsig, \du)\in Q_h\times \check Q_{h}\times V_h$ be the solution of \eqref{XGW}. There exists
\begin{equation} 
\|\dsig\|_{\rm div, h} + \|\du\|_{0, h} +\|\hsig\|_{0, h}\lesssim \|f\|_0.
\end{equation}
If $\csig\in H^{k+2}(\Om, \S)$, $u\in H^{k+1}(\Om, \R^n) ( k\ge0 )$, let $(\dsig, \hsig, \du)\in Q_h^{k+1}\times \check Q_{h}^k \times V_h^{k}$ be the solution of \eqref{XGW}, then we have the following error estimate:
\begin{equation}\label{div_error}
\|\csig-\dsig\|_{\rm div, h}  + \|u-\du\|_{0, h} + \|\hsig\|_{0, h}\lesssim h^{k+1}(|\csig|_{k+2} + |u|_{k+1}).
\end{equation}
Furthermore, if $k\ge n$,
\begin{equation}\label{div_error2}
\|\csig - \dsig\|_A\lesssim h^{k+2}(|\csig|_{k+2} + |u|_{k+1}).
\end{equation}
\end{enumerate}
\end{theorem} 
\subsubsection{A special case of three-field formulation without $\hsig$}
For each variable $\batau= (\dtau, \htau)\in Q_h\times \check  Q_h$, define the weak divergence $\divh_w: Q_h\times \check  Q_{h} \rightarrow V_h$ by
\begin{equation}
(\divh_w \batau, w_h)_{0, K}=-(\epsilon_h ( w_h), \dtau)_{0, K} + \langle (\{\dtau\} + \htau)\bn, w_h\rangle_{0, \partial K},\ \forall w_h\in V_h.
\end{equation}
The following lemma presents the relation between a special three-field formulation \eqref{XGW} and the weak Galerkin method.
\begin{lemma}
The formulation \eqref{XGW} with $\eta_1=4\eta_2$, $\gamma=0$, $Q_h\bn|_\cE\subset \check Q_{h}$ and $Q_h\bn|_\cE\subset \check V_{h}$ is equivalent to  the problem that finds $\basig\in Q_h\times \check Q_{h}$ and $\du\in V_h$ such that 
\begin{equation}\label{XGWG}
\left\{
\begin{array}{rll}
(A\dsig,\dtau) + (\divh_w \batau, \du)
+ s(\basig, \batau)
&=0,
&\batau\in Q_h \times \check Q_{h},
\\
(\divh_w \basig, \dv)
&
=(f,\dv),
&\dv\in V_h
\end{array}
\right.
\end{equation}
with $s(\basig, \batau)= \langle  2\eta_2(\tsig - \dsig)\bn, (\ttau - \dtau)\bn\rangle_{\PcT}$ and $\tu$ and $\tsig$ defined in \eqref{uWG} and \eqref{SigHDG}, respectively. 
\end{lemma}
\subsubsection{Hybridization for the three-field  formulation \eqref{XGWG}}
Denote
$$
Z_h = \{\du\in V_h: \epsilon_h (\du)=0\},
$$
$$
V_h^\perp = \{\du\in V_h: (\du, \dv)=0,\ \forall \dv\in Z_h\}.
$$
{ For any $\tsig\in \check Q_h$, denote $\hat\csig_{h,n}|_e=\tsig\bn_e$ and $\hat\csig_{h,t}|_e=\tsig t_e$ where $t_e$ is the unit tangential vector of edge $e$.}
By \eqref{errusig},  the three-field formulation \eqref{XGWG} can be decomposed into two sub-problems as:
\begin{enumerate}
\item[(I)] Local problems. For each element $K$, given ${\hat\csig_{h, n} \in  \hat  Q_{h}\bn}$, find $(\dsig^K, \du^K)\in Q_h\times V_h^\perp$ such that for any $(\dtau, \dv)\in Q_h\times V_h^\perp$
\begin{equation}\label{XGWG1}
\left\{
\begin{array}{rll}
(A\dsig^K,\dtau)_K
- (\epsilon_h (\du^K), \dtau)_K
+\langle  2\eta_2\dsig^K \bn, \dtau\bn\rangle_{\partial K}
&
=\langle  2\eta_2{\hat\csig_{h, n}}, \dtau\bn\rangle_{\partial K},
\\
-(\dsig^K, \epsilon_h (\dv))_K
&
=(f,\dv)_K-\langle  {\hat\csig_{h, n}}, \dv\rangle_{\partial K}.
\end{array}
\right.
\end{equation}
It is easy to see that the local problem  \eqref{XGWG1} is well-posed if $\epsilon_h (V_h)\subset Q_h$.
Denote $W_Q: {\check  Q_h\bn}\rightarrow Q_h$ and $W_V: {\check  Q_h\bn}\rightarrow V_h^\perp$ by 
$$ 
W_Q({\hat\csig_{h, n}})|_K= \dsig^K\quad\text{ and }\quad  W_V({\hat\csig_{h, n}})|_K= \du^K,
$$
respectively. 
\item[(II)] Global problem. Find $\tsig$ such that $({\hat\csig_{h, n}}, \du^0)\in \hat  Q_{h}\times Z_h$ satisfies
\begin{equation} \label{XGWG2}
\left\{
\begin{array}{rll}
\langle  { \hat\csig_{h, n}}, \dv^0\rangle_\PcT,
&
=(f,\dv^0),&\ \forall \dv^0\in Z_h,
\\
\langle 2\eta_2 ({\hat\csig_{h, n}} - W_Q({\hat\csig_{h, n}})\bn) +  W_V({\hat\csig_{h, n}}) +\du^0, {\hat\ctau_{h,n}}\rangle_\PcT
&= 0,&\ \forall {\hat\ctau_{h,n}} \in {\check  Q_{h}n},
\end{array}
\right.
\end{equation}
and { $\hat\csig_{h, t}|_\cE=(\{W_Q({\hat\csig_{h, n}})\} -\eta_1^{-1}[W_V({\hat\csig_{h, n}})])t|_\cE$}.
It follows from \eqref{XGWG1} that
\begin{equation}
\begin{split}
 (AW_Q({\hat\csig_{h, n}}), W_Q({\hat\ctau_{h,n}}))_K
- (\epsilon_h (W_V({\hat\csig_{h, n}})), W_Q({\hat\ctau_{h,n}}))_K&
\\
=\langle   2\eta_2({\hat\csig_{h, n} }-W_Q({\hat\csig_{h, n}})\bn),& W_Q({\hat\ctau_{h,n}})\bn\rangle_{\partial K},
\\
\langle  W_V({\hat\csig_{h, n}}), \ttau\bn\rangle_{\partial K}
- \langle  W_Q({\hat\ctau_{h,n}}), \epsilon_h(W_V({\hat\csig_{h, n}}))\rangle_{\partial K}
&=
(f, W_V({\hat\csig_{h, n}}))_K.
\end{split}
\end{equation}
Thus the second equation in \eqref{XGWG2} can be written as  
\begin{equation}
\begin{split}
\langle \eta_2 ({\hat\csig_{h, n}} - W_Q({\hat\csig_{h, n}})\bn), {\hat\ctau_{h,n}}-W_Q({\hat\ctau_{h,n}})\bn \rangle_\PcT
+\langle \du^0, W_V(\hat\ctau_{h,n}) \rangle_\PcT
= - (f, W_V(\hat\ctau_{h,n})).
\end{split}
\end{equation}
Therefore, the global sub-problem \eqref{XGWG2} seeks $\tsig$ where $(\hat \csig_{h, n}, \du^0)\in \hat  Q_{h}\times Z_h$ 
\begin{equation} \label{XGWG3}
\left\{
\begin{array}{rll}
\langle \eta_2 (\hat \csig_{h, n} - W_Q(\hat \csig_{h, n})\bn), \hat\ctau_{h, n}-W_Q(\hat\ctau_{h,n})\bn \rangle_\PcT
+\langle \du^0,  \hat\ctau_{h, n} \rangle_\PcT
&
=
- (f, W_V(\hat\ctau_{h,n})),
\\
\langle \hat \csig_{h, n}, \dv^0\rangle_\PcT
&
=(f,\dv^0),
\end{array}
\right.
\end{equation}
{ for any $(\hat\ctau_{h,n}, \dv^0) \in \check  Q_{h}n\times  Z_h$, and 
 $\hat\csig_{h, t}|_\cE=(\{W_Q(\hat\csig_{h, n})\} -\eta_1^{-1}[W_V(\hat\csig_{h, n})])t|_\cE$}.
\end{enumerate}

Note that  the  three-field formulation 
is hybridizable under the Assumptions (G1)--(G3) or (D1)--(D3). This implies that  the corresponding four-field formulation \eqref{XGgrad} or \eqref{XGdiv} is hybridizable. 

\subsection{Two-field formulation without the variables $\hsig$ and $\hu$}
{ Recall the two-field formulation  \eqref{LDGXG}  seeks: $(\dsig, \du)\in Q_h \times V_h $ such that }
\begin{equation}\label{XGDG}
\left\{
\begin{array}{rlr}
(A\dsig,\dtau)
+(\du, \divh_h \dtau)
-\langle \tu, \dtau\bn\rangle_{0, \partial K}
&
=0,&\ \forall \dtau \in Q_h,
\\
-(\dsig, \epsilon_h (\dv))
+\langle   \tsig\bn, \dv\rangle_{0, \partial K},
&
=(f,\dv),&\ \forall \dv\in V_h,
\end{array}
\right.
\end{equation}
with
\begin{equation}
\begin{array}{rlr}
\tsig|_e &= \check P_h^\sigma ( \{\dsig\}- \tau [\du] + [\dsig]\gamma^T)\quad &\text{on}\  \cE,\\
\tu|_e &=  \check P_h^u (\{\du\} - \eta [\dsig] - { (\gamma^T\bn)[\du]\bn }) \quad &\text{on}\ \cE.
\end{array}
\end{equation}
It is a
generalization of DG methods \cite{chen2010local,cockburn2000development,arnold2002unified}.

\begin{theorem}\label{Th:D}
 There exist the following properties:
\begin{enumerate}
\item  Under Assumptions (G1)--(G3), the $H^1$-based formulation \eqref{XGDG} is uniformly well-posed with respect to 
mesh size, $\rho_1$ and $\rho_2$.  Let $(\dsig,  \du)\in Q_h \times V_h$ be the solution of \eqref{XGDG}. There exists
\begin{equation} 
\|\dsig\|_{0, h}+ \|\du\|_{1, h} \lesssim \|f\|_{-1,h}.
\end{equation}
If $\csig\in H^{k+1}(\Om, \S)$, $u\in H^{k+2}(\Om, \R^n) ( k\ge 0 )$, let $(\dsig,   \du)\in Q_h^k  \times V_h^{k+1}$ be the solution of \eqref{XGDG}, then we have the following error estimate:
\begin{equation} 
\|\csig-\dsig\|_{0, h} + \|u-\du\|_{1, h}\lesssim h^{k+1}(|\csig|_{k+1} + |u|_{k+2}).
\end{equation}
\item Under Assumptions (D1)--(D3), the $H({\rm div})$-based formulation \eqref{XGDG} is uniformly well-posed with respect to 
mesh size, $\rho_1$ and $\rho_2$.  Let $(\dsig, \hsig, \du)\in Q_h\times \check Q_{h}\times V_h$ be the solution of \eqref{XGDG}. There exists
\begin{equation} 
\|\dsig\|_{\rm div, h} + \|\du\|_{0, h}\lesssim \|f\|_0
\end{equation}
If $\csig\in H^{k+2}(\Om, \S)$, $u\in H^{k+1}(\Om, \R^n) ( k\ge0 )$, let $(\dsig,  \du)\in Q_h^{k+1}\times  V_h^{k}$ be the solution of \eqref{XGDG}, then we have the following error estimate:
\begin{equation}\label{div_error}
\|\csig-\dsig\|_{\rm div, h}  + \|u-\du\|_{0, h} \lesssim h^{k+1}(|\csig|_{k+2} + |u|_{k+1}).
\end{equation}
Furthermore, if $k\ge n$,
\begin{equation}\label{div_error2}
\|\csig - \dsig\|_A\lesssim h^{k+2}(|\csig|_{k+2} + |u|_{k+1}).
\end{equation}
\end{enumerate}
\end{theorem} 

Table \ref{tab:LDG1} lists some well-posed $H^1$-based methods and the second method is a new one. 
It shows that the LDG method in \cite{chen2010local} is the first one in Table \ref{tab:LDG1} with $k=1$, $\eta=\gamma=0$ and $\tau=O(h_e^{-1})$. 
%The last formulation in Table \ref{tab:LDG1} is a brand new LDG method. It is equivalent to the second three-field formulation in Table \ref{tab:HDGexist} and the first one in Table \ref{tab:WG}. 
The comparison between the methods in Table \ref{tab:LDG1} implies that the vanishing parameter $\eta$ causes the failure of the hybridization for the method in \cite{chen2010local}.

\begin{table}[!htbp] 
\centering
\begin{tabular}{c|cccccccccc}
\hline
cases&$\eta$& $\tau$& $\gamma$&$Q_h$  & $\check{Q}_h$& $V_h$ &$\check V_h$&
\\ 
\hline
1&0& $\Omega(h_e^{-1})$&0&$ Q_h^k$ & $\check{Q}_h^k$& $V_h^{k+1}$& $\check{V}_h^{k+1}$&\cite{chen2010local}
\\
2&$\mathcal{O}(h_e)$& $\mathcal{O}(h_e^{-1})$&$\mathcal{O}(1)$&$ Q_h^{k}$ & $\check{Q}_h^{k}$& $V_h^{k+1}$& $\check{V}_h^{k}$   & new 
%3&$\tau^{-1}$& $\Omega(h_e^{-1})$&$0$&$ Q_h^{k}$ & $\check{Q}_h^{k}$& $V_h^{k+1}$& $\check{V}_h^{k}$   &new
\\\hline
\end{tabular}
\caption{\footnotesize{$H^1$-based methods for linear elasticity problem.}} 
\label{tab:LDG1}
\end{table}
%\begin{table}[!htbp]
%\footnotesize
%\centering
%\begin{tabular}{c|cccccccccc}
%\hline
%cases&$\eta$& $\tau$& $\gamma$&$Q_h$  & $\check{Q}_h$& $V_h$ &$\check V_h$&$\|\csig-\dsig\|_0$ & $\|\epsilon (u)-\epsilon_h (\du)\|_0$&
%\\ 
%\hline
%1&0& $\Omega(h_e^{-1})$&0&$ Q_h^k$ & $\check{Q}_h^k$& $V_h^{k+1}$& $\check{V}_h^{k+1}$ & $k+1$ & $k+1$  &\cite{chen2010local}
%\\
%2&$\mathcal{O}(h_e)$& $\mathcal{O}(h_e^{-1})$&$\mathcal{O}(1)$&$ Q_h^{k}$ & $\check{Q}_h^{k}$& $V_h^{k+1}$& $\check{V}_h^{k}$ & $k+1$ &$k+1$  & new \\
%3&$\tau^{-1}$& $\Omega(h_e^{-1})$&$0$&$ Q_h^{k}$ & $\check{Q}_h^{k}$& $V_h^{k+1}$& $\check{V}_h^{k}$ & $k+1$ &$k+1$  &new
%\\\hline
%\end{tabular}
%\caption{$H^1$-based formulations for linear elasticity problem.} 
%\label{tab:LDG1}
%\end{table}
 
Table \ref{tab:LDG2} lists the LDG method in \cite{wang2020mixed} and some new $H({\rm div})$-based methods. With the same choice of parameters and discrete spaces, all these methods are well-posed and admit the optimal error estimates for both the displacement and the stress tensor.
It shows that the method induced from the formulation \eqref{XGDG} with $\tau=0$, $\gamma=0$ and $\eta=O(h_e^{-1})$ is equivalent to the LDG method in \cite{wang2020mixed}. 
The last two cases in Table \ref{tab:LDG2} are brand new LDG methods.  
It implies that the vanishing parameter $\tau$ causes the failure of the hybridization for the method in \cite{wang2020mixed}.

\begin{table}[!htbp] 
\centering
\begin{tabular}{c|ccccccccccc}
\hline
cases&$\eta$& $\tau$&$\gamma$& $Q_h$  & $\check{Q}_h$& $V_h$ &$\check V_h$& 
\\ 
\hline
1&$\Omega(h_e^{-1})$& 0&0& $ Q_h^{k+1}$  &
$\check{Q}_h^{k}$& $V_h^{k}$&$\check{V}_h^{k+1}$  &\cite{wang2020mixed}
\\
2&$\mathcal{O}(h_e^{-1})$& $\mathcal{O}(h_e)$& $\mathcal{O}(1)$&$ Q_h^{k+1}$  &
$\check{Q}_h^{k+1}$& $V_h^{k}$&$\check{V}_h^{k+1}$   & new
\\
3&$\tau^{-1}$& $\Omega(h_e)$& $0$&$ Q_h^{k+1}$  &
$\check{Q}_h^{k+1}$& $V_h^{k}$&$\check{V}_h^{k+1}$   &new
\\
\hline 
\end{tabular}
\caption{\footnotesize{$H({\rm div})$-based methods for linear elasticity problem.}}
\label{tab:LDG2}
\end{table} 
%\begin{table}[!htbp]
%\footnotesize
%\centering
%\begin{tabular}{c|ccccccccccc}
%\hline
%cases&$\eta$& $\tau$&$\gamma$& $Q_h$  & $\check{Q}_h$& $V_h$ &$\check V_h$&$\|\csig-\dsig\|_0$ & $\|\divh_h (\csig-\dsig)\|_0$ & $\|u-\du\|_0$&
%\\ 
%\hline
%1&$\Omega(h_e^{-1})$& 0&0& $ Q_h^{k+1}$  &
%$\check{Q}_h^{k}$& $V_h^{k}$&$\check{V}_h^{k+1}$ & $k+1$ &$k+1$ &$k+1$ &\cite{wang2020mixed}
%\\
%2&$\mathcal{O}(h_e^{-1})$& $\mathcal{O}(h_e)$& $\mathcal{O}(1)$&$ Q_h^{k+1}$  &
%$\check{Q}_h^{k+1}$& $V_h^{k}$&$\check{V}_h^{k+1}$ & $k+1$ &$k+1$ &$k+1$ & new
%\\
%3&$\tau^{-1}$& $\Omega(h_e)$& $0$&$ Q_h^{k+1}$  &
%$\check{Q}_h^{k+1}$& $V_h^{k}$&$\check{V}_h^{k+1}$ & $k+1$ &$k+1$ &$k+1$ &new
%\\
%\hline 
%\end{tabular}
%\caption{$H({\rm div})$-based formulations for linear elasticity problem.} 
%\label{tab:LDG2}
%\end{table} 
 
%In some cases, a one-field formulation can be obtained from \eqref{XGDG} following the idea in \cite{arnold2002unified} for elliptic problems.
%By the first equation in \eqref{XGDG}, the variable $\dsig$ can be represented in terms of the variable $\du$. 
%Thus, the two-field formulation \eqref{XGDG} can be further reduced to a one-field formulation with only one variable $\du$.

%\input{../qingguo}

\section{Two limiting cases}\label{sec:limit}

\subsection{Mixed methods: A limiting case of the formulation \eqref{XGHDG0}}
The mixed methods \cite{gopalakrishnan2011symmetric,hu2014family,hu2015family,arnold2002mixed}  for linear elasticity problems can be generalized into the following formulation which seeks $(\dsig^M, \du^M)\in Q_{h}^{M}\times V_{h}$ such that
\begin{equation}\label{mixEq}
\left\{
\begin{array}{rll}
(A\dsig^M,\dtau^M) 
+(\du^M, \divh \dtau^M) 
&=0,&\ \forall \dtau^M\in Q_{h}^M,
\\
(\divh_h\dsig^M, \dv) 
&=(f,\dv),&\ \forall \dv\in V_{h},
\end{array}
\right.
\end{equation}
with 
\begin{equation*}
Q_{h}^{M}=\{\dtau\in Q_h: \langle [\dtau], \hv \rangle=0, \ \forall \hv\in \check V_{h}\}.
\end{equation*}
%Let $Q_h=Q_h^{k+1}$, $V_h=V_h^{k}$, $\check V_h=\check V_h^{k}$ for any $k\ge 0$, the above formulation \eqref{mixEq} turns to be  the nonconforming element in \cite{gopalakrishnan2011symmetric}.  
Let $Q_h=Q_h^{k+1}$, $V_h=V_h^{k}$, $\check V_h=\check V_h^{k+1}$ for any $k\ge n$, {the formulation \eqref{mixEq}} becomes the 
conforming mixed element in \cite{hu2014family,hu2015family}. Let  
$$
Q_h=\{\dtau\in Q_h^{k+2}, \divh_h \dtau|_K\in P_{k}(K, \mathbb{R}^2)\},\quad V_h=V_h^{k},\quad \check V_h=\check V_h^{k+2}
$$ 
for any $k\ge 1$. The corresponding formulation \eqref{mixEq} is the conforming mixed element in \cite{arnold2002mixed}.

Consider the three-field formulation \eqref{XGH} with $\gamma=0$, $\tau_2=0$, $\check Q_h=\{0\}$ and $V_h|_\cE \subset \check{V}_h$. 
By the DG identity \eqref{DGidentity}, this three-field formulation  seeks $(\dsig, \du,  \hu)\in Q_h \times V_h \times \check{V}_{h}$ such that for any $(\dtau, \dv, \hv)\in Q_h \times V_h \times \check V_{h}$,
\begin{equation}\label{HDGmixed}
\left\{
\begin{split}
(A\dsig,\dtau) 
+(\du, \divh_h \dtau) 
-\langle \hu+\{\du\},[\dtau]\rangle
&= 0,
\\
(\divh_h\dsig, \dv) 
- \langle [\dsig], \{ \dv\}\rangle
&=(f,\dv),
 \\
\langle [\dsig], \hv \rangle &=0,
\end{split}
\right.
\end{equation}
which is equivalent  to the mixed formulation \eqref{mixEq}.  
 { As stated in Remark \ref{remarkdiv}, the three-field formulation \eqref{HDGmixed} is well-posed, thus \eqref{mixEq} is also well-posed with}
%By Theorem \ref{Th:inf-supDiv}, the three-field formulation \eqref{HDGmixed} is uniformly well-posed under Assumptions (D1)--(D3). Thus, the limiting formulation \eqref{mixEq} is also well-posed with
\begin{equation}\label{mix:stability}
\|\dsig^M\|_{\rm div, h} + \|\du^M\|_{0, h}\lesssim \|f\|_0.
\end{equation}

Furthermore, a similar analysis to the one in \cite{hong2020extended} {provides} the following theorem.
\begin{theorem}\label{th:mix}
Assume (D1)-(D3) hold. 
Let $(\dsig,   \du)\in Q_h \times V_h$ be the solution of \eqref{LDGXG} and 
$(\dsig^M, \du^M)\in Q_{h}^{M}\times V_{h}$ be the  solution of the corresponding mixed method \eqref{mixEq}. 
If $V_h|_\cE\subset \check V_h$, the formulation \eqref{LDGXG} with $\gamma=0$ and $\rho_1 + \rho_2\rightarrow 0$ converges to the  
mixed method \eqref{mixEq} and 
\begin{equation}
\|\dsig - \dsig^M\|_0+\|\divh_h(\dsig - \dsig^M)\|_0 + \|\du - \du^M\|_0 \lesssim  ( \rho_1^{\frac{1}{2}} + \rho_2^{\frac{1}{2}}) \|f\|_0 .
\end{equation}
\end{theorem}
\begin{proof}
Recall the two-field formulation \eqref{LDGXG}
\begin{equation}\label{LDGXG1}
\left\{
\begin{array}{rlr}
(A\dsig,\dtau)
+\langle   \tau_2^{-1} \check{P}_h^u [\dsig], [\dtau]\rangle + (\divh_h\dtau,\du)
-\langle  [\dtau] , \{\du\}\rangle
&=0,
&\forall \dtau\in Q_h,
\\
(\divh_h\dsig,\dv)
-\langle  [\dsig] , \{\dv\}\rangle
- \langle   \tau_1 \check{P}_h^\sigma[\du]\bn, [\dv]\bn\rangle
&=(f,\dv),
&\forall \dv \in V_h.
\end{array}
\right.
\end{equation}
Substracting \eqref{mixEq} from \eqref{LDGXG1}, we obtain 
\begin{equation}\label{XGerroreq}
\left\{
\begin{array}{rlr}
(A(\dsig-\dsig^M),\dtau)
+\langle  \tau_2^{-1} \check{P}_h^u [\dsig-\dsig^M], [\dtau]\rangle &+ (\divh_h\dtau,\du-\du^M)
-\langle  [\dtau] , \{\du-\du^M\}\rangle\\
=  -(u_h^M, {\rm div}_h (\bm \tau_h-\bm \tau_h^M))&-(A\bm \sigma_h^M, \bm \tau_h-\bm \tau_h^M)+\langle  [\dtau] , \{\du^M\}\rangle 
\\
(\divh_h(\dsig-\dsig^M),\dv)
-\langle  [\dsig-\dsig^M] , \{\dv\}\rangle &- \langle\tau_1 \check{P}_h^\sigma[\du-\du^M]\bn, [\dv]\bn\rangle
\\
&= \langle\tau_1 \check{P}_h^\sigma[\du^M]\bn, [\dv]\bn\rangle 
\end{array}
\right.
\end{equation}
for any $(\dtau, \dv)\in Q_h\times V_h$. 
By the stability estimate in Theorem \ref{Th:D}, trace inequality and note that $\tau_1=\rho_1 h_e $, $\tau_2=\rho_2 h_e$,
\begin{equation}\label{error:stability1}
\begin{aligned}
&\|\dsig-\dsig^M\|_{\rm div, h} + \|\du-\du^M\|_{0, h}
\\
\lesssim &\sup_{\dtau\in Q_h}\frac{
 { |-(u_h^M, {\rm div}_h (\bm \tau_h-\bm \tau_h^M))-(A\bm \sigma_h^M, \bm \tau_h-\bm \tau_h^M)+\langle  [\dtau] , \{\du^M\}\rangle | }}{\|\dtau\|_{\rm div, h}}
\\
&+\sup_{\dv\in V_h}\frac{ { |\langle\tau_1 \check{P}_h^\sigma[\du^M]\bn, [\dv]\bn\rangle |} }{\|\dv\|_{0, h}}\\
\lesssim &\sup_{\dtau\in Q_h}\frac{
\|u_h^M\|_0 \|{\rm div}_h (\bm \tau_h-\bm \tau_h^M)\|_0 
+ \|{ A}\bm \sigma_h^M\|_0 \|\bm \tau_h-\bm \tau_h^M\|_0}{\|\dtau\|_{\rm div, h}}+( \rho_1^{\frac{1}{2}} + \rho_2^{\frac{1}{2}})\|\du^M\|_{0}
\end{aligned}
\end{equation}
where $\|\cdot \|_{\rm div, h}$ and $\|\cdot\|_{0, h}$ are defined in \eqref{divnormdef}.
%\begin{equation}
%\begin{array}{ll}
%\|\dtau\|_{\rm div, h}^2 =(A\dtau,\dtau)
%+ \|\divh_h \dtau\|_0^2
%+ \| \tau_2^{-1/2}[\dtau ]\|_0^2,\quad
%\|\dv\|_{0, h}^2 =\|\dv\|_0^2+ \|\tau_1^{1/2}[\dv]\|_0^2+\|\tau_2^{1/2}\{\dv\}\|_0^2.
%\end{array}
%\end{equation} 

For any given $\bm \tau_h \in Q_h$, we have 
\begin{equation}
\begin{aligned}
\inf_{\bm \tau_h^M\in Q_h^M} \left(\|{\rm div}_h (\bm \tau_h-\bm \tau_h^M)\|+ \|\bm \tau_h-\bm \tau_h^M\|\right)\lesssim 
\big(\sum_{e\in \mathcal{E}_h} h_e^{-1}\|[\bm
\tau_h]\|^2_{0,e}\big)^{\frac{1}{2}}\le \rho_2^{\frac12}  \|\bm \tau_h\|_{{\rm div},h}
\end{aligned}
\end{equation}
It follows from stability estimates \eqref{mix:stability} that
\begin{equation}\label{error:stability2}
\|\dsig-\dsig^M\|_{\rm div, h} + \|\du-\du^M\|_{0, h}\lesssim ( \rho_1^{\frac{1}{2}} + \rho_2^{\frac{1}{2}}){ \left(\|\du^M\|_{0}+  \|\bm \sigma_h^M\|_0\right)}\lesssim ( \rho_1^{\frac{1}{2}} + \rho_2^{\frac{1}{2}})\|f\|_0,
\end{equation} 
which completes the proof.
\end{proof}

\subsection{Primal methods: A limiting case of the formulation \eqref{XGDG}} 
The  primal method for linear elasticity problems seeks $\du^P\in V_{h}^{P}$ such that
\begin{equation}\label{non:primal}
(C\epsilon_h(\du^P), \epsilon_h(\dv))=-(f,\dv),\ \forall \dv\in V_{h}^{P}
\end{equation}
with $C=A^{-1}$ and
\begin{equation}
V_{h}^{P}=\{\du\in V_h: \langle [\du], \htau \rangle= 0, \forall \htau\in \check{Q}_h\},
\end{equation}
where $[\dv]$ is defined in \eqref{jumpdef}.
If $\epsilon_h(V_h)\subset Q_h$, { the formulation \eqref{non:primal} }is equivalent to the following formulation which seeks $(\dsig^P, \du^P)\in Q_{h}\times V_{h}^{P}$ such that
\begin{equation}\label{primalEq}
\left\{
\begin{array}{rll}
(A\dsig^P,\dtau)
-(\dtau, \epsilon_h (\du^P))
&=0,&\dtau\in Q_{h},
\\
-(\dsig^P, \epsilon_h (\dv))
&=(f,\dv),
&\dv\in V_{h}^P.
\end{array}
\right.
\end{equation} 
Consider the three-field formulation \eqref{XGW} with $\gamma=0$, $\check V=\{0\}$  seeks $(\dsig, \du,  \hsig)\in Q_h \times V_h \times \check{Q}_{h}$ such that
\begin{equation}\label{XGconf}
\left\{
\begin{array}{rll}
(A\dsig,\dtau)
-(\dtau, \epsilon_h (\du))
+\langle  \{\dtau\}\bn, [\du]\bn\rangle 
&= 0,&\dtau\in Q_h,
\\
-(\dsig, \epsilon_h (\dv))
+\langle  \{\dsig\}\bn  + \hsig\bn , [\dv]\bn\rangle
&
=(f,\dv),  &\dv\in V_{h},
\\
\langle \eta_1  \hsig , \htau \rangle
+ \langle  [\du], \htau \rangle
&=0,& \htau \in \check{Q}_{h}.
\end{array}
\right.
\end{equation}
If $V_h|_\cE \subset \check Q_h\bn$, as $\eta_1\rightarrow 0$,  the resulting formulation  is exactly the primal formulation \eqref{primalEq}. 
Under the assumptions (G1)-(G3),  Theorem \ref{Th:inf-supGrad} implies the well-posedness of   \eqref{XGconf}, and 
\begin{equation} 
\|\dsig\|_{0, h} + \|\du\|_{1, h} +\|\hsig\|_{0, h}\lesssim \|f\|_0 .
\end{equation}
{ By Remark \ref{remarkgrad}, the primal formulation \eqref{primalEq}
}
%Since this estimate is uniformly bounded with respect to $\rho_1$, the primal  formulation 
is also well-posed with
\begin{equation} 
\|\dsig^P\|_{0, h}+\|\du^P\|_{1, h}\lesssim 
\sup_{\dv\in V_h^P} \frac{(f,\dv)}{\|\dv\|_{1, h}}.
\end{equation}
\begin{remark}
If $V_h=V_h^{k+1}, Q_h=Q_h^k, \check Q_h=\check Q_h^k, k\ge 1$, the { formulation \eqref{XGconf}} tends to a high order nonconforming discretization 
\eqref{non:primal} for the elasticity problem with only one variable.  The relationship between the Crouzeix-Raviart  element discretization and discontinuous Galerkin method for 
linear elasticity can be found in \cite{hansbo2003discontinuous}.
\end{remark}

In addition, a similar analysis to the one of Theorem \ref{th:mix} { provides} the following theorem. 
\begin{theorem}
Assume that (G1)-(G3) hold. Let $(\dsig, \du)\in Q_h \times V_h $ be the solution of \eqref{XGDG} and $(\dsig^P, \du^P)\in Q_{h}\times V_{h}^{P}$ be the solution of the corresponding primal method \eqref{primalEq}.  Then the formulation \eqref{XGDG} with   $\rho_1 + \rho_2\rightarrow 0$ converges to the primal method \eqref{primalEq} and
\begin{equation}
\|\dsig - \dsig^P\|_0 + \|\epsilon_h (\du - \du^P)\|_0 +  \|h_e^{-1/2}[\du-\du^P]\|_0
\lesssim   (\rho_1^{1/2} + \rho_2^{1/2})  \|f\|_0. 
\end{equation} 
\end{theorem}

\section{Numerical examples}\label{sec:numerical}
 
In this section, we display some numerical experiments in 2D to verify the estimate provided in Theorem \ref{Th:inf-supGrad}  and \ref{Th:inf-supDiv}. We consider the model problem \eqref{model} on the unit square $\Om=(0,1)^2$ with the exact displacement 
$$
u= (\sin (\pi x)\sin(\pi y), \sin (\pi x)\sin(\pi y))^T,
$$
and set $f$ and $g$ to satisfy the above exact solution of \eqref{model}.
The domain is partitioned by uniform triangles. The level one triangulation $\mathcal{T}_1$ consists of two right triangles, obtained by cutting the unit square with a north-east line. Each triangulation $\mathcal{T}_i$ is refined into a half-sized triangulation uniformly, to get a higher level triangulation $\mathcal{T}_{i+1}$. For all the numerical tests in this section, fix the parameters $\rho_1=\rho_2=\gamma=1$ and $E=1$.  

{\subsection{Various methods with fixed $\nu$}
In this subsection, we fix $\nu=0.4$. }
Figure \ref{fig:lowgrad1} and \ref{fig:lowgrad} plot the errors for the lowest order $H^1$-based methods mentioned in this paper. 
Figure \ref{fig:lowgrad1} and \ref{fig:lowgrad}  show that the $H^1$-based XG methods with 
$$
Q_h=Q_h^0, V_h=V_h^1, \check Q_h=\check Q_h^0, \check V_h=\check V_h^i\quad \mbox{with}\quad i=0, 1
$$ 
are not well-posed, while those with 
$$
Q_h=Q_h^0, V_h=V_h^1, \check Q_h=\check Q_h^1, \check V_h=\check V_h^i\quad \mbox{with}\quad i=0, 1
$$ 
are well-posed and admit the optimal convergence rate 1.00 as analyzed in Theorem \ref{Th:inf-supGrad}. The discrete spaces of the former methods satisfy Assumption (G1),  but does not meet Assumption (G2). This implies that Assumption (G2) is necessary for the wellposedness of the corresponding method.

\begin{figure}[H]
\setlength{\abovecaptionskip}{0pt}
\setlength{\belowcaptionskip}{0pt}
\centering
\includegraphics[width=7cm]{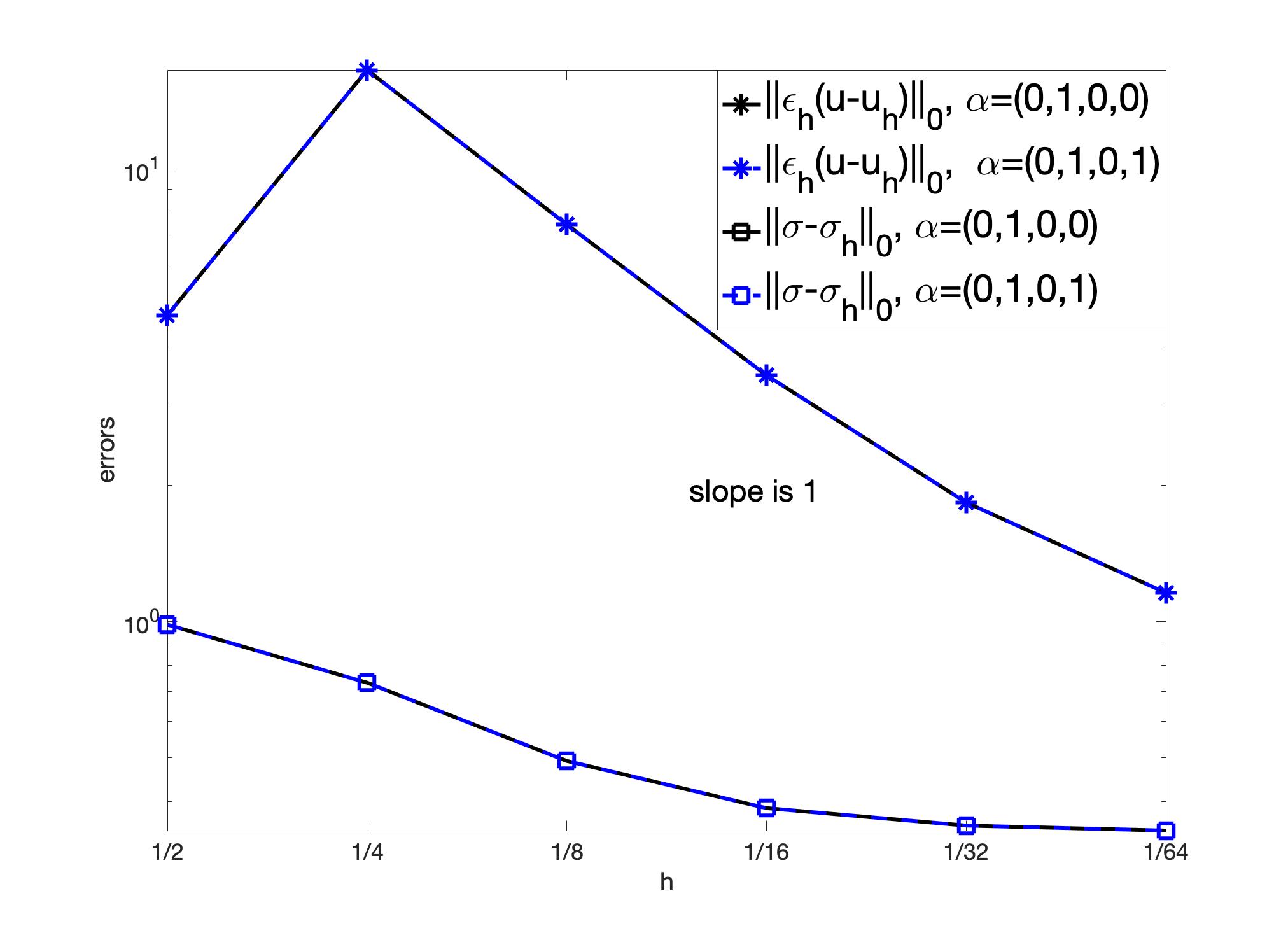} 
\caption{\small Errors of the lowest order $H^1$-based methods with $
Q_h=Q_h^{\alpha_1}$, $V_h=V_h^{\alpha_2}$, $\check Q_h=\check Q_h^{\alpha_3}$, $\check V_h=\check V_h^{\alpha_4}$ and $\alpha=(\alpha_1, \alpha_2, \alpha_3, \alpha_4)$.}
\label{fig:lowgrad1}
\end{figure}

\begin{figure}[H]
\setlength{\abovecaptionskip}{0pt}
\setlength{\belowcaptionskip}{0pt}
\centering 
\includegraphics[width=7cm]{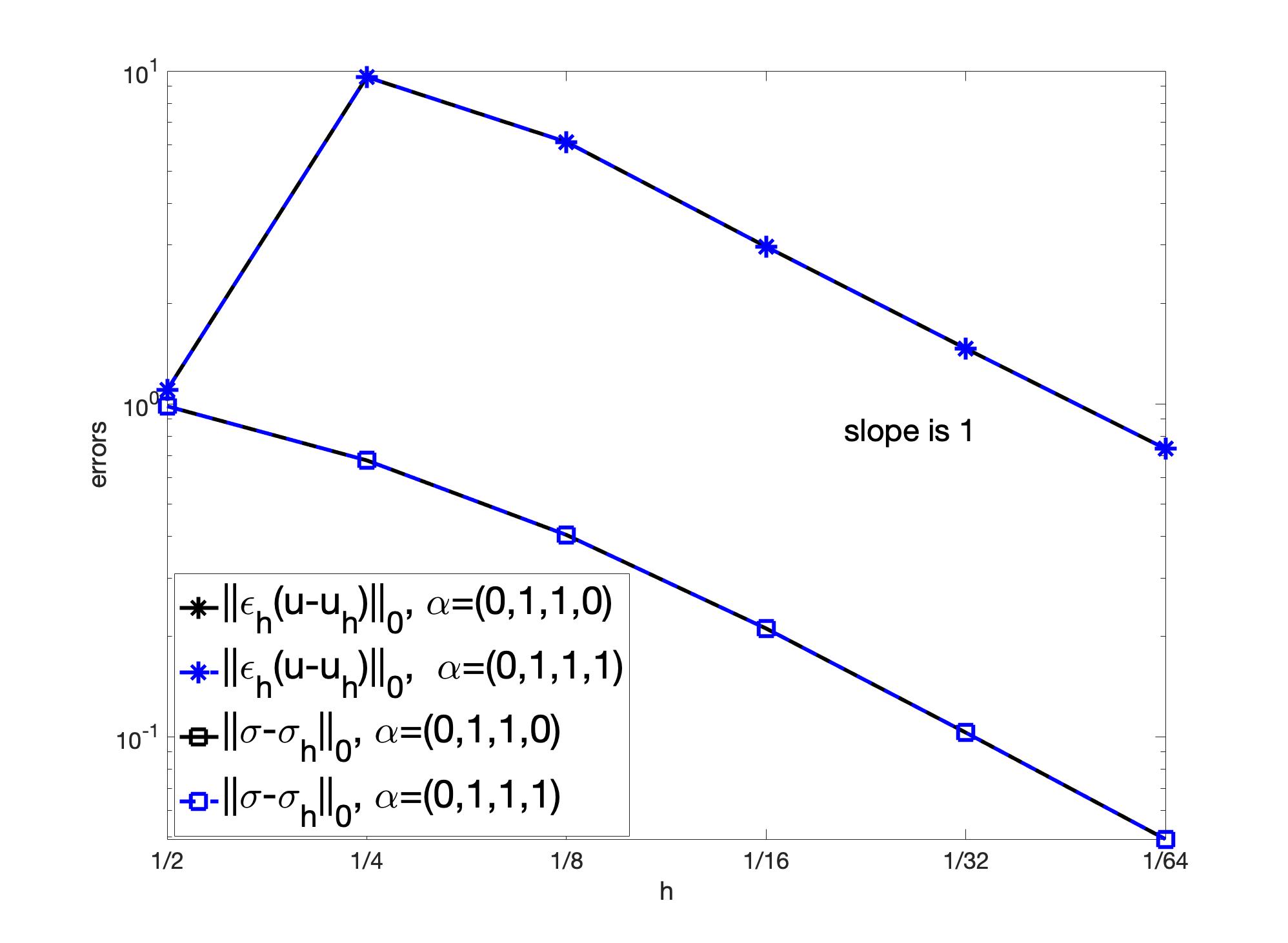} 
\caption{\small Errors of the lowest order $H^1$-based methods with $
Q_h=Q_h^{\alpha_1}$, $V_h=V_h^{\alpha_2}$, $\check Q_h=\check Q_h^{\alpha_3}$, $\check V_h=\check V_h^{\alpha_4}$ and $\alpha=(\alpha_1, \alpha_2, \alpha_3, \alpha_4)$.}
\label{fig:lowgrad}
\end{figure}

\begin{figure}[H]
\setlength{\abovecaptionskip}{0pt}
\setlength{\belowcaptionskip}{0pt}
\centering 
\includegraphics[width=7cm]{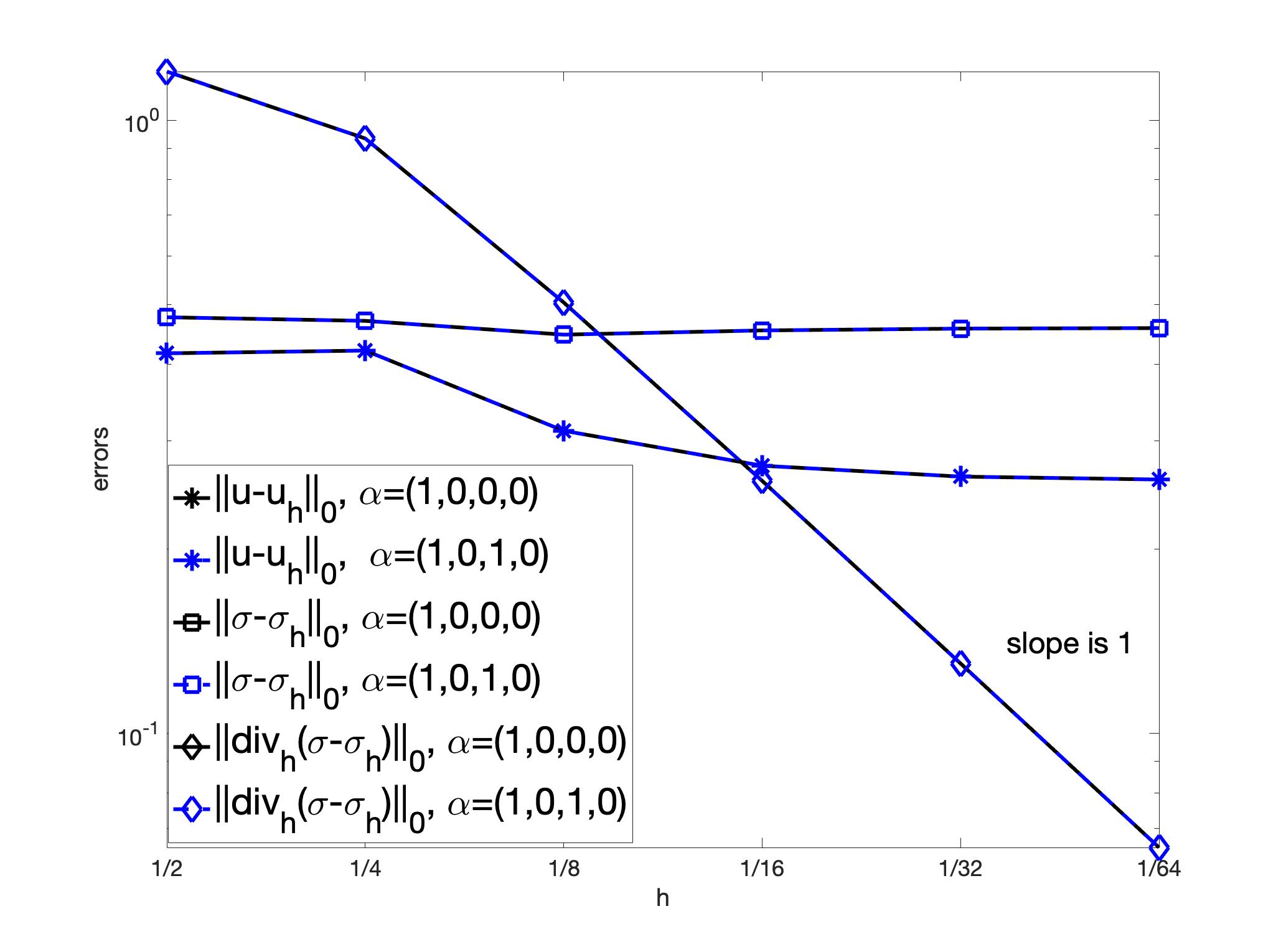} 
\caption{\small Errors of the lowest order $H({\rm div})$-based methods with $
Q_h=Q_h^{\alpha_1}$, $V_h=V_h^{\alpha_2}$, $\check Q_h=\check Q_h^{\alpha_3}$, $\check V_h=\check V_h^{\alpha_4}$ and $\alpha=(\alpha_1, \alpha_2, \alpha_3, \alpha_4)$.}
\label{fig:lowdiv1}
\end{figure}

Figure \ref{fig:lowdiv1} and \ref{fig:lowdiv}  plot the errors of solutions for the lowest order $H({\rm div})$-based methods, which are new in literature. It is shown that the $H({\rm div})$-based methods with 
\begin{equation}\label{test:choice0}
Q_h=Q_h^{1}, V_h=V_h^0, \check Q_h=\check Q_h^{i}, \check V_h=\check V_h^0\quad \mbox{with}\quad 0\le i\le 1
\end{equation}
are not well-posed. Although the error $\|{\rm div}_h(\csig-\dsig)\|_0$ converges at the optimal rate $1.00$, the errors $\|\csig - \dsig\|_0$ and $\|u-\du\|_0$ do not converge at all. It also shows in Figure  \ref{fig:lowdiv1} and \ref{fig:lowdiv} that the { new  lowest order $H({\rm div})$-based methods} with a larger space for $\check V_h$
\begin{equation}\label{test:choice}
Q_h=Q_h^{1}, V_h=V_h^0, \check Q_h=\check Q_h^{i}, \check V_h=\check V_h^{1}\quad \mbox{with}\quad 0\le i\le 1
\end{equation}
are well-posed and the corresponding errors $\|\csig - \dsig\|_0$, $\|\divh_h(\csig - \dsig)\|_0$ and $\|u-\du\|_0$ admit the optimal convergence rate $1.00$.  This coincides with the results in Theorem \ref{Th:inf-supDiv}. The comparison between the  $H({\rm div})$-based methods  in \eqref{test:choice0} and \eqref{test:choice} implies that Assumption (D2) is necessary for the wellposedness of the corresponding method.

\begin{figure}[H]
\setlength{\abovecaptionskip}{0pt}
\setlength{\belowcaptionskip}{0pt}
\centering  
\includegraphics[width=7cm]{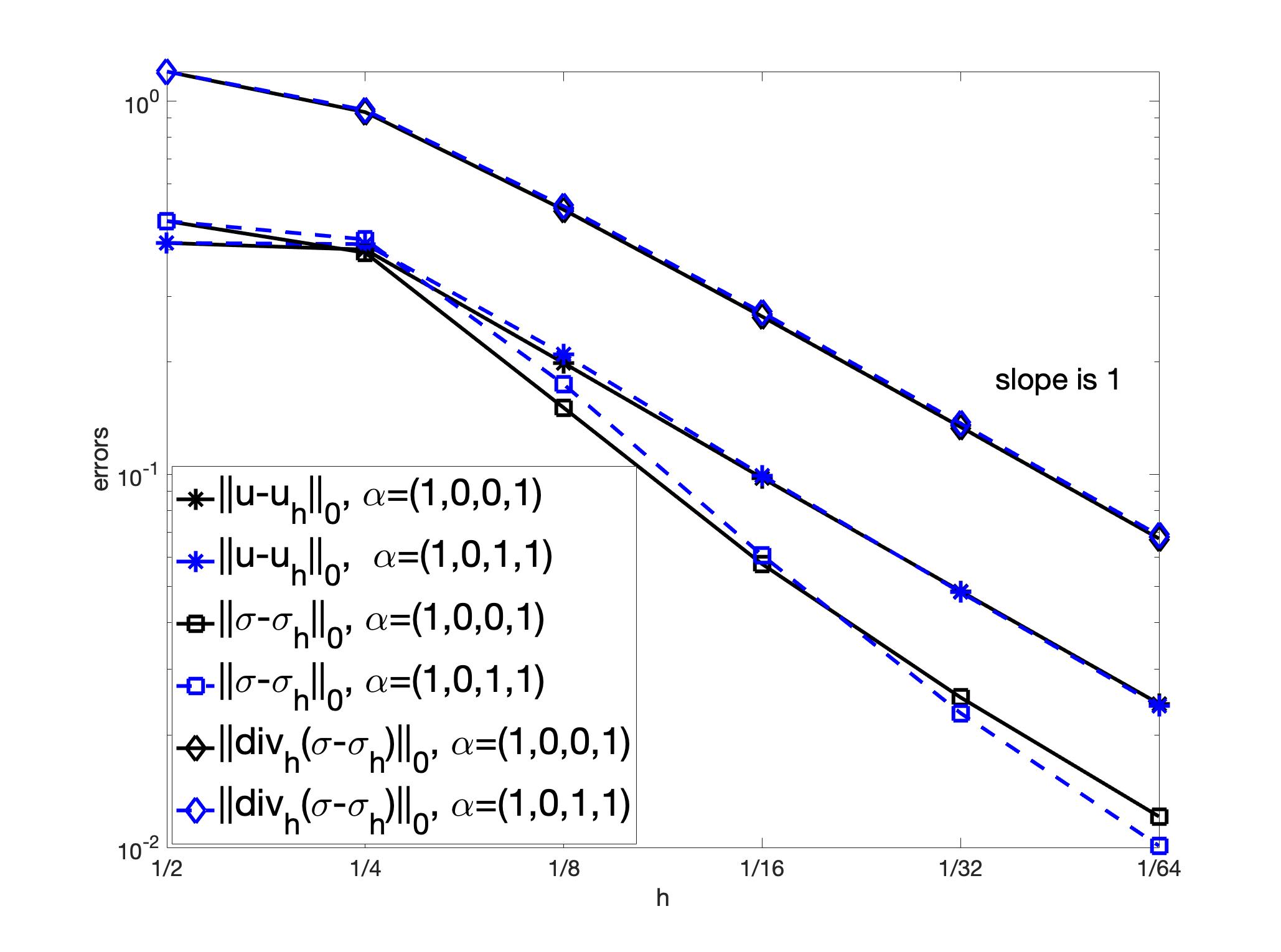}
\caption{\small Errors of the lowest order $H({\rm div})$-based methods with $
Q_h=Q_h^{\alpha_1}$, $V_h=V_h^{\alpha_2}$, $\check Q_h=\check Q_h^{\alpha_3}$, $\check V_h=\check V_h^{\alpha_4}$ and $\alpha=(\alpha_1, \alpha_2, \alpha_3, \alpha_4)$.}
\label{fig:lowdiv}
\end{figure}

\begin{figure}[!ht]
\setlength{\abovecaptionskip}{0pt}
\setlength{\belowcaptionskip}{0pt}
\centering   
\includegraphics[width=7cm]{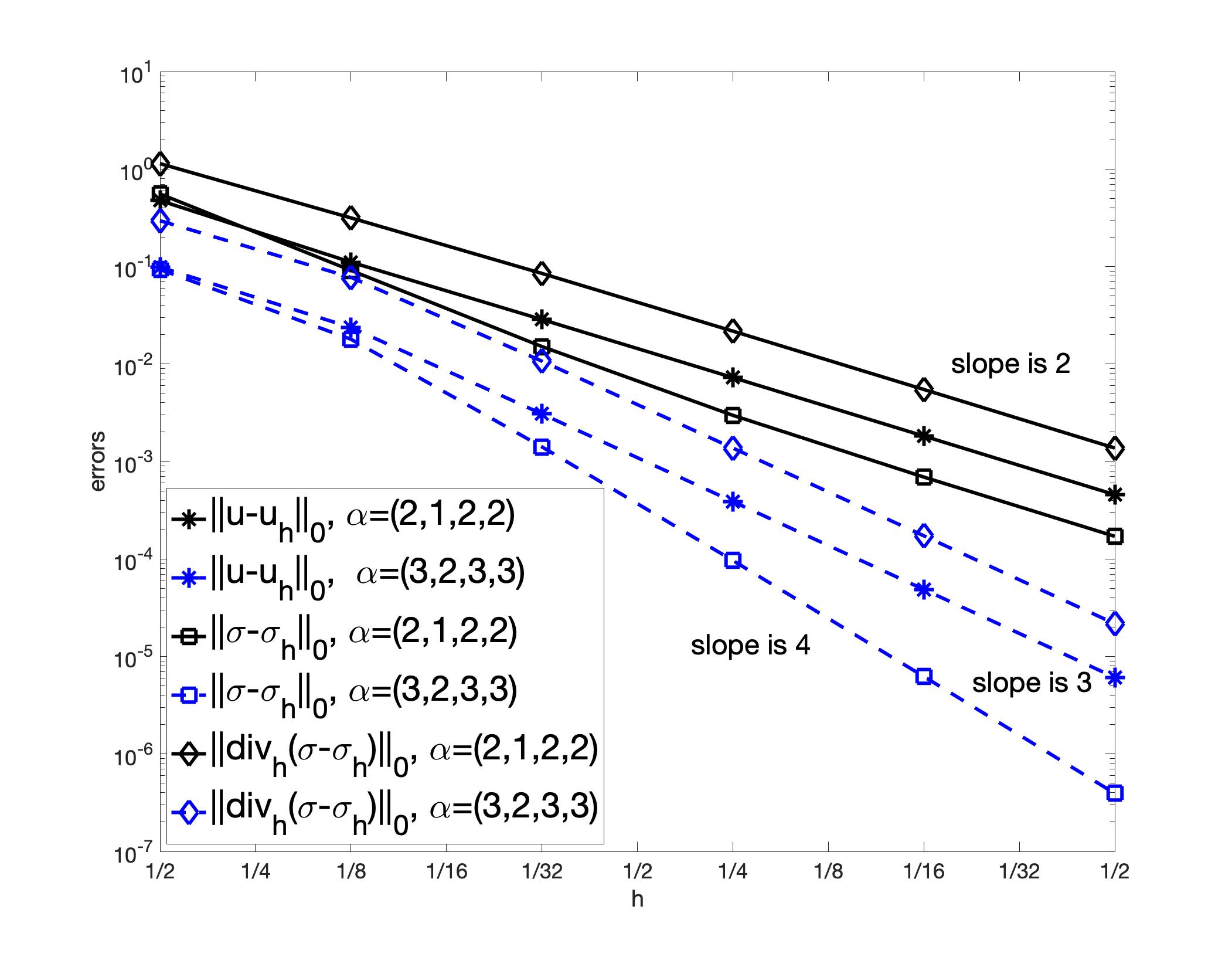} 
\caption{\small Errors for some high order $H({\rm div})$-based methods with $
Q_h=Q_h^{\alpha_1}$, $V_h=V_h^{\alpha_2}$, $\check Q_h=\check Q_h^{\alpha_3}$, $\check V_h=\check V_h^{\alpha_4}$ and $\alpha=(\alpha_1, \alpha_2, \alpha_3, \alpha_4)$.}
\label{fig:div}
\end{figure}
 
Consider the $L^2$ norm of the error of the stress tensor $\csig$. Figure \ref{fig:div} plots the errors for higher order $H({\rm div})$-based methods. For the XG formulation with $\alpha=(2, 1, 2, 2)$, $k=1$ is less than $n=2$. Theorem \ref{Th:inf-supDiv} indicates that the convergence rate of $\|\csig - \dsig\|_0$ is $2.00$ { for this new second order $H({\rm div})$-based method}, which is verified by the numerical results in Figure \ref{fig:div}. For the XG formulation with $\alpha=(3, 2, 2, 3)$, $k=n$ and the convergence rate of $\|\csig - \dsig\|_0$ shown in Figure \ref{fig:div} is $4$, which coincides with the estimate in Theorem \ref{Th:sigL2div}. This comparison indicates that the assumption $k\ge n$ in Theorem \ref{Th:sigL2div} is necessary and the error estimate of $\|\csig - \dsig\|_0$ is optimal. 

{\subsection{The lowest order methods with various $\nu$}

Figure \ref{fig:gradlam} plots the relative error of the approximate solutions of the $H^1$-based method with $Q_h^{0}\times \check Q_{h}^1\times V_h^{1}\times \check V_{h}^{0}$ with different $\nu$ ($\nu$ tends to $\frac12$). Figure \ref{fig:gradlam} shows that both $\|\epsilon_h(u -\du)\|_0$ and $\|\csig - \dsig\|_0$ converge at the optimal rate 1.00, and the error $\|\csig - \dsig\|_0$ are almost the same for different value of $\nu$.

\begin{figure}[!ht]
\setlength{\abovecaptionskip}{0pt}
\setlength{\belowcaptionskip}{0pt}
\centering   
\includegraphics[width=7cm]{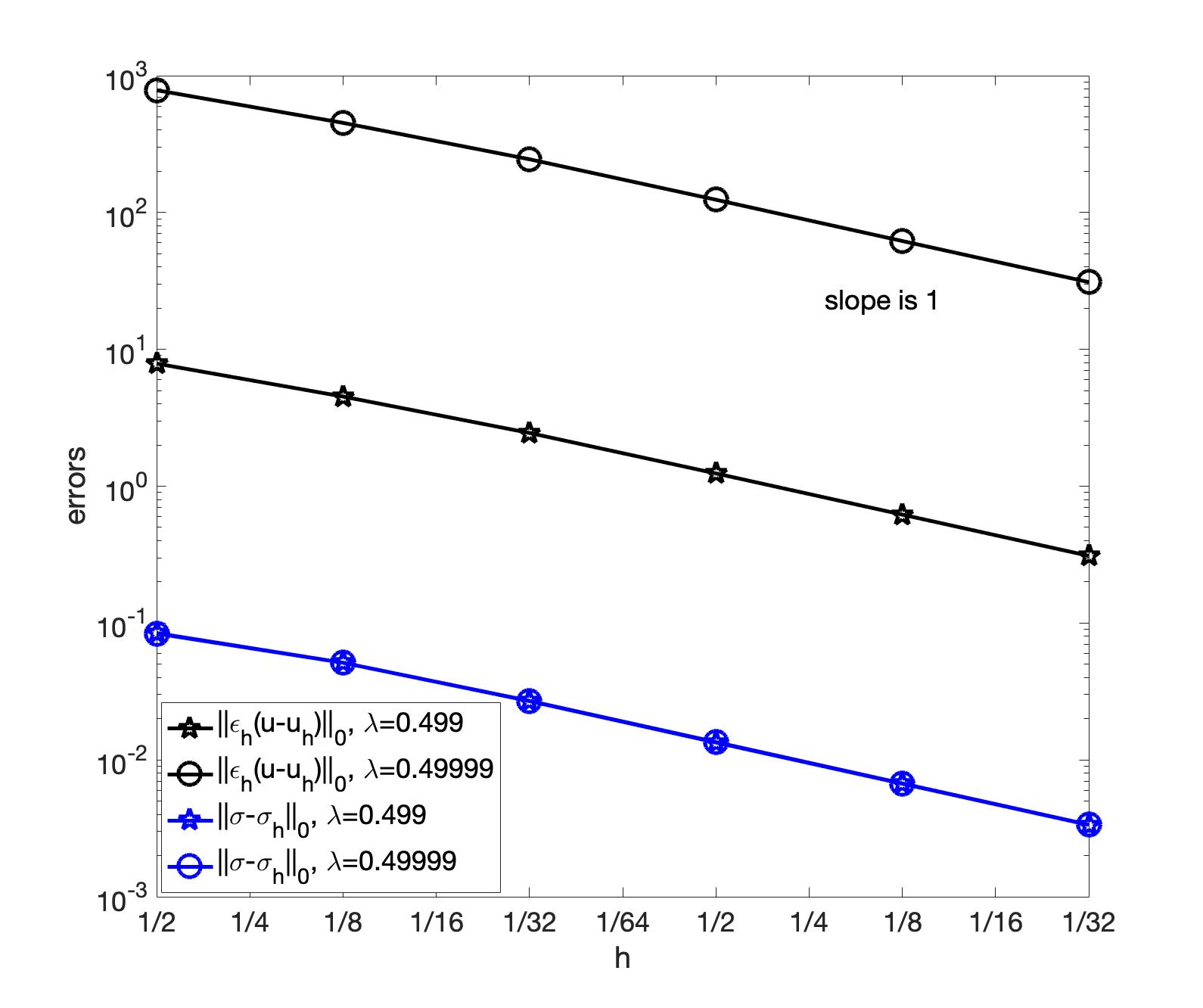} 
\caption{\small Errors for the lowest order $H^1$-based methods $Q_h^{0}\times \check Q_{h}^1\times V_h^{1}\times \check V_{h}^{0}$ with different $\nu$.}
\label{fig:gradlam}
\end{figure}

Figure \ref{fig:gradlam} plots the relative error of the approximate solutions of the $H({\rm div})$-based method with $Q_h^{1}\times \check Q_{h}^0\times V_h^{0}\times \check V_{h}^{1}$ with different $\nu$ ($\nu$ tends to $\frac12$). Figure \ref{fig:divlam} shows that the errors $\|u -\du\|_0$, $\|\csig - \dsig\|_0$ and $\|\divh(\csig - \dsig)\|_0$ converge at the optimal rate 1.00, and the errors $\|\csig - \dsig\|_0$ and $\|\divh(\csig - \dsig)\|_0$ are almost the same as $\nu$ tends to $\frac12$ which shows that the proposed schemes are locking-free.

\begin{figure}[!ht]
\setlength{\abovecaptionskip}{0pt}
\setlength{\belowcaptionskip}{0pt}
\centering   
\includegraphics[width=7cm]{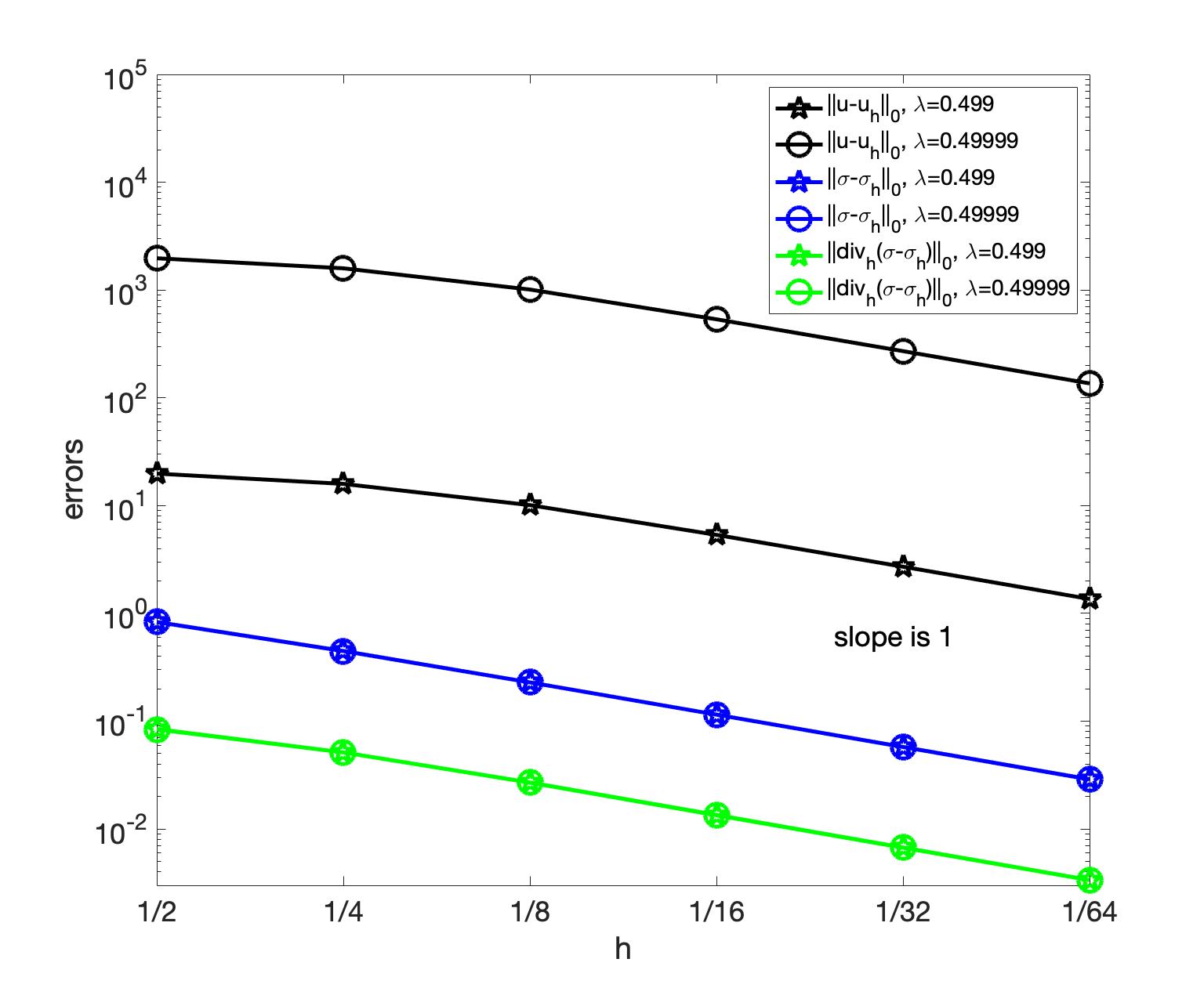} 
\caption{\small Errors for the lowest order $H({\rm div})$-based methods $Q_h^{1}\times \check Q_{h}^0\times V_h^{0}\times \check V_{h}^{1}$ with different $\nu$.}
\label{fig:divlam}
\end{figure}
}

\section{Conclusion}\label{concl}
In this paper, a unified analysis of a four-field formulation is presented and analyzed  for linear elasticity problem. 
This formulation is closely related to most HDG methods, WG methods, LDG methods and mixed finite elements in the literature. 
And some new methods are proposed following the unified framework. Some particular methods are proved to be hybridizable.
In addition, uniform inf-sup conditions for the four-field formulation provide a 
unified way to prove the optimal error estimate under two different sets of assumptions. 
Also, these assumptions guide the design of some well-posed formulations new in literature.

\bibliographystyle{spmpsci}      % mathematics and physical sciences
\bibliography{../../bibelasticity}   % name your BibTeX data base
\end{document}